\documentclass{amsart}
\usepackage{amsmath, amssymb}
\usepackage{bbm}
\usepackage{latexsym}
\usepackage{xcolor}
\usepackage{tikz}
\usetikzlibrary{arrows}
\usetikzlibrary{patterns}
\usetikzlibrary{shapes}
\usepackage{mathrsfs}
\usepackage{bbm}
\usepackage{amsfonts}

\usepackage{tikz-cd}

\newcommand{\R}{\mathbb{R}}
\newcommand{\Z}{\mathbb Z}
\newcommand{\N}{\mathbb N}

\tikzset{>=stealth',every on chain/.append style={join},
         every join/.style={->}}
 \usepackage[usenames,dvipsnames]{pstricks}
 \usepackage{epsfig}
 \usepackage{pst-grad} 
 \usepackage{pst-plot} 

\newtheorem{thm}{Theorem}
\newtheorem{lemma}[thm]{Lemma}
\newtheorem{cor}[thm]{Corollary}
\newtheorem{defi}[thm]{Definition}
\newtheorem{prop}[thm]{Proposition}

\newtheorem*{thmi}{Theorem}

\theoremstyle{remark} 
\newtheorem{remark}[]{Remark}

\newcommand{\be}{\begin{equation}}
\newcommand{\ee}{\end{equation}}
\newcommand{\weakto} {\rightharpoonup}

\newcommand{\argmin} {{\rm argmin}}

\usepackage{mathtools} 
\mathtoolsset{showonlyrefs, showmanualtags} 

\title[Homotopy properties of horizontal loop spaces]{Homotopy properties of horizontal loop spaces and applications to closed sub-riemannian geodesics}
\author{Antonio Lerario}   \thanks{Antonio Lerario, SISSA. email: \textsf{lerario@sissa.it}.}
\author{Andrea Mondino}   \thanks{Andrea Mondino, The University of Warwick, Mathematics Institute. email: \textsf{a.mondino@warwick.ac.uk}.}

\begin{document}
\maketitle
\begin{abstract}
Given a manifold $M$ and a proper sub-bundle $\Delta\subset TM$, we investigate homotopy properties of the \emph{horizontal}   free loop space $\Lambda$, i.e. the space of absolutely continuous maps $\gamma:S^1\to M$ whose velocities are constrained to $\Delta$ (for example: legendrian knots in a contact manifold).

In the first part of the paper we prove that the base-point map $F:\Lambda \to M$ (the map associating to every loop its base-point) is a Hurewicz fibration for the $W^{1,2}$ topology on $\Lambda$. Using this result we show that, even if  the space $\Lambda$ might have deep singularities (for example: constant loops form a singular manifold homeomorphic to $M$),  its homotopy can be controlled nicely. In particular we prove that $\Lambda$ (with the $W^{1,2}$ topology) has the homotopy type of  a CW-complex, that its inclusion in the standard   free loop space (i.e. the space of loops with no non-holonomic constraint) is a homotopy equivalence, and consequently  its homotopy groups can be computed as $\pi_k(\Lambda)\simeq \pi_k(M) \ltimes \pi_{k+1}(M)$ for all $k\geq 0.$

In the second part of the paper we address the problem of the existence of \emph{closed} sub-riemannian geodesics. In the general case we prove that if $(M, \Delta)$ is a compact sub-riemannian manifold, each non trivial homotopy class in $\pi_1(M)$ can be represented by a closed sub-riemannian geodesic. 

In the \emph{contact} case, we prove a min-max result generalizing the celebrated Lyusternik-Fet theorem: if $(M, \Delta)$ is a compact, contact manifold, then every sub-riemannian metric on $\Delta$ carries at least one closed sub-riemannian geodesic. This result is based on a combination of the above topological results with the delicate study of an analogous of a Palais-Smale condition in the vicinity of \emph{abnormal} loops (singular points of $\Lambda$).
\end{abstract}

\section{Introduction}
\subsection{The horizontal loop space}In this paper we study the topology of the space of loops $\gamma:S^1\to M$ whose velocities are constrained in a \emph{non-holonomic} way (we call these loops \emph{horizontal}). The constraint is made explicit by requiring that the loops should be absolutely continuous curves (hence differentiable almost everywhere) and that their velocity should belong a.e. to a totally non-integrable distribution $\Delta\subset TM.$

The case $\Delta=TM$ clearly imposes no constraint. The case when $\Delta$ is integrable imposes a constraint which is still holonomic (loops are confined on leaves of a foliation, by Frobenius Theorem) and can be reduced to the previous one. The totally non-integrable (or non-holonomic) case arises by requiring that the given distribution satisfies the \emph{H\"ormander} condition: a finite number of iterated brackets of smooth sections of $\Delta$ should span all the tangent space $TM$ (see \cite{AgrachevBarilariBoscain, Montgomery}). Contact manifolds are probably the most well known non-holonomic geometries and their (smooth) horizontal loops are called \emph{legendrian knots} \cite{Sabloff}.

In this paper we require our loops to have square-integrable velocity; this requirement determines a natural topology on the loop space, as follows. Consider first the set of all horizontal paths:
\be \Omega=\{\gamma:[0,1]\to M\,|\, \textrm{$\gamma$ is absolutely continuous, $\dot{\gamma}\in \Delta$ a.e. and is $L^2$-integrable}\}.\ee
The space $\Omega$ endowed with the $W^{1,2}$ topology is a Hilbert manifold modeled on $L^{2}(I, \R^d)\times \R^m$ (where $d=\textrm{rank}(\Delta)$ and $m=\textrm{dim}(M)$) and the \emph{endpoint map} is the smooth function:
\be F:\Omega\to M\times M,\quad \gamma\mapsto(\gamma(0), \gamma(1)).\ee
The object of our interest, the \emph{free horizontal loop space}, will thus be defined as:
\be \Lambda=F^{-1}(\textrm{diagonal in $M\times M$})\hookrightarrow \Omega, \ee
and endowed with the induced topology.

In the non-holonomic case, the loop space $\Lambda$ is a highly \emph{singular} object. For example, constant loops form a whole singular manifold (homeomorphic to $M$ itself). In the \emph{contact} case these are the ``only'' singularities (Proposition \ref{prop:singularities} below), but in general the presence of \emph{abnormal} curves might imply other (deep) singularities, see \cite{AgrachevBarilariBoscain, CJT} and \cite[Section 5]{Montgomery}. The structure of these singularities is at the origin of Liu and Sussmann's minimality theorem \cite{LiuSuss} (see the proof of this theorem given in \cite{AgrachevBarilariBoscain} and the discussion in \cite[Section 3.8]{Montgomery}).

\begin{remark} The uniform convergence topology  on $\Omega$ has been studied in \cite{Sarychev} and the $W^{1,1}$ in \cite{dynamic}. The $W^{1,p}$ topology with $p>1$ has been investigated by the first author and F. Boarotto in \cite{BoarottoLerario} - for the scopes of calculus of variations the case $p>1$ is especially interesting as one can apply classical techniques from critical point theory to many problems of interest. All these topologies are equivalent from the point of view of homotopy theory \cite[Theorem 5]{BoarottoLerario}, but in the $W^{1,\infty}$ topology the so-called \emph{rigidity} phenomenon appears: some curves might be isolated (up to reparametrization), see \cite{Bryant}.
\end{remark}

\subsection{Homotopy properties of the horizontal loop space} 


One of the main technical ingredients in order to understand the topological structure of the horizontal loop space $\Lambda$ is the \emph{Hurewicz fibration} property for the endpoint map. Recall that a map between topological spaces is a (Hurewicz) fibration if it has the homotopy lifting property with respect to any space (see Section \ref{sec:hurewicz} below and \cite{rure, Spanier} for more details). Our first result proves this property for the endpoint map restricted to the loop space (in the following statement we identify $M$ with the diagonal in $M\times M$, where $F|_{\Lambda}$ takes values). The techniques for the proof use a novel combination of quantitative control theory and classical homotopy theory.

\begin{thmi}[The Hurewicz fibration property]
The map $F|_{\Lambda}:\Lambda\to M$, that associates to every horizontal loop its base-point, is a Hurewicz fibration for the $W^{1,2}$ topology on $\Lambda$.
\end{thmi}

As a consequence of this property, strong information on the homotopy of the horizontal loop space can be deduced. What is remarkable at this point is that, even if $\Lambda$ might be extremely singular, its homotopy is very well controlled.
\begin{thmi}[The homotopy of the horizontal loop space] The horizontal loop space $\Lambda$ has the homotopy type of a CW-complex and its inclusion in the standard loop space is a homotopy equivalence; in particular\footnote{The product in \eqref{eq:pikLambda} is \emph{semi-direct} only possibly for $k=1$; in all other case the homotopy group on the left is abelian and the product is indeed direct; see Remark \ref{rmk:semi} below.}
 for all $k\geq 0$:
\be\label{eq:pikLambda}
 \pi_{k}(\Lambda)\simeq \pi_k(M) \ltimes \pi_{k+1}(M).
 \ee
\end{thmi}

One immediate but remarkable corollary of the fact that $\Lambda$ has the homotopy type of a CW-complex is that  \emph{every} loop has a neighborhood which is contractible in $\Lambda$ (Corollary \ref{cor:contractible} below). This gives a remarkable sharpening  of the local structure of $\Lambda$ near a singular curve.

\begin{remark}[The legendrian fundamental group]
Notice that the homotopy groups of the standard loop space have  the same structure  \eqref{eq:pikLambda} of the horizontal ones. This fact is  in sharp contrast with the situation of  \emph{legendrian regular} homotopies of  legendrian knots in a contact manifold: given a contact manifold $(M,\Delta)$,  a $C^1$ horizontal loop $\gamma:S^1 \to M$ which is also an immersion is called a \emph{legendrian knot}. Two legendrian knots are \emph{legendrian-homotopically} equivalent if there exists a $C^1$ homotopy between them all made of legendrian knots.  The \emph{legendrian} fundamental group is then defined as the group whose elements are the equivalence classes of legendrian knots under the equivalence relation of being legendrian-homotopically equivalent.  By using the $h$-principle it is possible to show \cite[Section 3.3]{EES} that there is a surjection of the legendrian fundamental group into the standard fundamental group, and   (for sake of simplicity, here we consider dim$(M)$=3)
 the kernel of the surjection is $\Z$.  Roughly, this means that it is highly not true in general that two horizontal knots  inducing the same class in $\pi_0(\Lambda)\simeq \pi_1(M)$ are also legendrian-homotopically equivalent. This is due to the much stronger constraint, in the legendrian setting, that the loop (and the homotopy at a fixed time) is an immersion. 
\end{remark}

\subsection{The calculus of variation on the horizontal loop space}
Once a sub-riemannian structure (i.e. a smooth scalar product) on $\Delta$  is fixed, on the space of horizontal paths we can define the length functional:
\be\label{eq:deflength}
\textrm{Length} (\gamma)=  \int_{[0,1]} |\dot{\gamma}(t)|\, dt, \quad \forall \gamma \in \Omega(M).
\ee
A horizontal loop is called a \emph{sub-riemannian closed geodesic} if  it has constant speed and it is locally length minimizing (see Definition \ref{def:closedGeod}).
On the horizontal path space we can also define the energy functional:
 \be\label{eq:defJint}
J(\gamma)=\frac 1 2  \int_{[0,1]} |\dot{\gamma}(t)|^2 \, dt, \quad \forall \gamma \in \Omega(M).
\ee
The two functionals \eqref{eq:deflength}-\eqref{eq:defJint}   are linked in a well known matter, and it is clear that the energy functional $J$ fits   better then the length with the $W^{1,2}$ topology on $\Omega$ (e.g. it is coercive and weakly lower semicontinuous). 

It should be clear that the singularities of the space $\Lambda$, which in our problem acts as a ``constraint'' for the the functional $J$, will cause serious  problems in extending the classical setup for finding critical points of $J|_{\Lambda}.$
We nevertheless notice (Proposition \ref{prop:CharClosedG}) that solutions of the  normal Lagrange multiplier equation for $J$ constrained to $\Lambda$ are indeed sub-riemannian geodesics (but the viceversa is false). Therefore in order to show existence of closed sub-riemannian  geodesics it will be enough to show existence of solutions to the previous equation (which locally can be written as in \eqref{eq:lagrange}). This will be done in two different ways: by minimization and by min-max.

\subsection{Closed sub-riemannian geodesics} We will first show that if  $\pi_1(M)$ is not trivial then there always exists a closed sub-riemannian geodesic, extending  the  celebrated theorem of Cartan  (proved in the riemannian framework) to arbitrary sub-riemannian structures.  This result will be achieved via a minimization process, well known in the literature  as ``direct method in the calculus of variations'', based on the compatibility of the  functional $J$ with the strong and weak $W^{1,2}$ topologies on $\Lambda$.
 
\begin{thmi}[Existence of closed sub-riemannian  geodesics in $\pi_1(M)$] \label{thm:pi1int}
Let $M$ be a compact, connected sub-riemannian manifold. Then for every \emph{nonzero} $\alpha\in \pi_1(M)$ there exists a closed sub-riemannian geodesic $\gamma:S^1\to M$ such that $[\gamma]=\alpha$ and $\gamma$ minimizes $J$ in its homotopy class.\end{thmi}
 
In case the manifold $M$ is simply connected, i.e. if $\pi_1(M)=\{0\}$, the existence of a closed geodesic is more subtle: indeed a minimization procedure would trivialize and give just a constant curve. To handle this case we then argue via  min-max: given $\alpha \in \pi_k(\Lambda)$, $k>1$, we define: 
\be\label{eq:defcaInt}
c_\alpha=\inf_{f \in \alpha} \sup_{\theta \in S^k} J(f(\theta)).
\ee
The goal is to prove that if  $c_\alpha>0$, then $c_\alpha$ is a critical value, i.e.  there exists a closed  geodesic with energy $c_\alpha$. 

The two needed technical  tools are the Palais-Smale property (which roughly says that if a sequence of ``approximate critical points'' weakly converges, then actually it converges strongly),  and the Deformation Lemma (if there are no critical values in an interval $[a,b]$, then it is possible to continuously deform the sub-level set corresponding to $b$ to the one corresponding to $a$).

For the min-max part of the paper, we will restrict ourselves to \emph{contact} sub-riemannian structures (see \cite[Section 6.1.2]{Montgomery} for more details). Under this assumption the space of horizontal loops $\Lambda$ is still singular, but the singularities correspond to constant loops.  From the technical viewpoint, the proof of the Palais-Smale property in the contact case  (see Theorem \ref{thm:PS}) is maybe the most challenging and original  part  of the paper, and it will be achieved by a sort of ``blow-up'' argument based on the choice of good coordinates.

Once the Palais-Smale property is settled, in order to  apply the   well known min-max techniques and get the existence of a closed geodesic, we still need to prove that the min-max level  $c_\alpha$  is strictly positive, under the assumption that $M$ is compact. In the riemannian case this part is quite straightforward, since small loops are contractible. In the sub-riemannian case this is less obvious (recall that the homotopy must be a horizontal loop for every time slice) but still true thanks to the fine topological properties of $\Lambda$ established in the first part of the paper (see in particular Theorem \ref{thm:epsilon} and Corollary \ref{cor:nonzero}).

The combination of all  the tools discussed so far  will allow us to prove the following theorem.

\begin{thmi}[Existence of closed sub-riemannian geodesics in contact manifolds]
Let $(M,\Delta)$ be a compact, contact sub-riemannian manifold.  Then there exists at least one non constant closed  sub-riemannian geodesic.
\end{thmi}
This result  is the counterpart, in contact geometry, of the celebrated Lyusternik-Fet Theorem  \cite{LF} asserting the existence of a closed geodesic in any compact riemannian manifold. The result of  Lyusternik-Fet (obtained in 1951) opened the door to an entire beautiful  research field in mathematics on the existence and multiplicity of closed geodesics in riemannian manifolds. It is our hope (and challenge for the future) that the present paper will serve as a solid  basis in order  to investigate these questions   also in the sub-riemannian setting.

\subsection{Structure of the paper}The first part of Section \ref{sec:general} below is devoted to the general constructions (mostly the study of the topology of the whole horizontal path space). In Section \ref{sec:global} we discuss the properties of what we call the ``global chart'', a useful technical device to reduce the problem when a set of vector fields generating the distribution is given. The Hurewicz fibration property is Theorem \ref{thm:hurewicz}, and the result on the homotopy of the loop space is a combination of Theorem \ref{thm:homotopy} and Theorem \ref{thm:i*iso}. Closed sub-riemannian geodesics are introduced in Section \ref{sec:geo} and the existence of a minimizer for each nonzero homotopy class in $\pi_1(M)$ is proved in Theorem \ref{thm:pi1}. Section \ref{sec:contact} is devoted to the contact case. The Palais-Smale property is discussed in  Section \ref{sec:PS} and the deformation Lemma and the min-max procedure in Section \ref{sec:min-max}. The existence of a closed sub-riemannian geodesic in the compact, contact case is proved in Theorem \ref{thm:ExMain}.

\subsection{Acknowledgements} 
Most of the research presented in this paper has been developed  while the first author was visiting the {\it Forschungsinstitut f\"ur Mathematik} at the ETH Z\"urich and  the second author was supported by ETH and  SNSF. The authors would like  to thank the Institut for the hospitality and the excellent working conditions. In the final steps of the revision, the second author was supported by the EPSRC First Grant EP/R004730/1.
\\The quality of this manuscript improved after the
revision made by the reviewer, whose careful reading has been extremely helpful
for the authors.

\section{General theory}\label{sec:general}
In this section we discuss the relationship between the topology of the manifold and the horizontal path space; the theory applies to sub-riemannian manifolds in general (we will restrict to contact manifolds only in the next section).	
	\subsection{The horizontal path space}
Let $M$ be a smooth manifold of dimension $m$  and $\Delta\subset TM$ be a smooth subbundle of rank $k$. The horizontal path space $\Omega$ is defined by:
\be \Omega=\{\gamma:[0,1]\to M\,|\, \textrm{$\gamma$ is absolutely continuous, $\dot{\gamma}\in \Delta$ a.e. and is $L^2$-integrable}\}.\ee
This definition requires the choice of a sub-riemannian structure on $\Delta$ in order to integrate the square of the norm of $\dot\gamma$, but the fact of being \emph{integrable} is independent of the chosen structure (we refer the reader to \cite{AgrachevBarilariBoscain, Montgomery} for more details).  In the following we set $I=[0,1]$.

\subsection{Topologies on the horizontal path space}
In the sequel we always make the assumption that the H\"ormander condition is satisfied and that the distribution $\Delta$ is bracket-generating (see \cite{AgrachevBarilariBoscain} for the geometrical implications of this assumption). 
For every $\gamma\in \Omega$ the vector bundle $\gamma^*\Delta$ is trivializable, hence there exist smooth time-dependent vector fields $X_1(\cdot, t), \ldots, X_k(\cdot, t)$ and an open neighborhood $W\subset M$ containing $\gamma(I)$ such that for every $t\in I$ and every $x\in W$:
\be \Delta_x=\textrm{span}\{X_1(x, t), \ldots, X_k(x,t)\}\ee
(if a sub-riemannian structure on $\Delta$ has been chosen we can pick these vector fields to be orthonormal for every $t\in I$).

Since $\gamma$ is horizontal, we can write:
\be \dot{\gamma}(t)=\sum_{i=1}^k v_i(t)\, X_i(\gamma(t), t) \quad \text{for a.e. } t\in [0,1], \quad \gamma(0)=y.\ee
Let $U_{(v,y)}=
V\times W\subset L^2(I, \R^k)\times M$ be an open set containing $(v_1, \ldots, v_k, y)$ such that for every $u=(u_1, \ldots, u_k)\in V$ and every $x\in W$ the solution $\gamma_{(u,x)}$ of the Cauchy problem
\be \label{eq:control} \dot x(t)=\sum_{i=1}^ku_i(t) \, X_i(x(t),t),\quad \text{for a.e. } t\in [0,1], \quad  x(0)=x,\ee
is defined at time $t=1$. The vector function $u$ is usually called the \emph{control}.

Consider the set of all curves $\{x:I\to M\}$ arising as solutions of $\eqref{eq:control}$ for $(u,x)\in U_{(v,y)}\subset L^2(I, \R^k)\times M $: we declare sets of this type to be a basis of open neighborhoods for a topology on $\Omega$. Note that  the correspondence $(u,x)\mapsto x(\cdot)$ is a homeomorphism (the fact that this map is one-to-one follows from \cite[Appendix E]{Montgomery}, this is the existence and uniqueness for the solution of \eqref{eq:control}). Clearly the topology induced on $\Omega$ depends  on the topology we have fixed on  $L^2(I, \R^k)$. 
If $L^2(I, \R^k)$ is given the strong topology, the resulting topology on the horizontal path space will be denoted by $\Omega^{1,2}$ and called the strong topology; if $L^2(I, \R^k)$ is endowed with the weak topology, the resulting topology will be denoted $(\Omega^{1,2})^{\textrm{weak}}$ and called the weak topology. Unless specified we will work with the strong topology.

\subsection{Hilbert manifold structure on the horizontal path space}
The choice of coordinates $\psi:W\to \R^m$ on a neighborhood in $M$  and of a (possibly time-dependent) trivializing frame $\mathcal{F}=\{X_1, \ldots, X_k\}$ for $\Delta|_{W}$ define an open chart on $\Omega$:
\be \Phi_{(\psi, \mathcal{F})}:U\to L^{2}(I, \R^k)\times \R^m.\ee
Here $U$ consists of those curves that are solutions of the Cauchy problem 
 $$\dot{\gamma}(t)=\sum_{i=1}^k u_i(t) \, X_i(\gamma(t),t),  \quad \text{for a.e. } t\in [0,1], \quad  \gamma(0)=x,$$
  for $(u,x)\in V\times W\subset L^2(I, \R^k)\times \R^m$. The fact that the collection $\{\Phi_{(\psi, \mathcal{F})}\}$ is a Hilbert manifold atlas for $\Omega$ follows from \cite[Appendix E]{Montgomery}.

The \emph{endpoint map} $F$ is defined by:
\be F:\Omega\to M\times M, \quad  \gamma\mapsto (\gamma(0), \gamma(1)).\ee	
Given a (locally) trivializing frame and a coordinate chart on an open set $U\simeq  L^{2}(I, \R^k)$ we denote by:
\be\label{defFxut}
F_x^t(u)=\text{ the solution at time $t$ of the Cauchy problem \eqref{eq:control} for a fixed $x\in W\subset M$.}
\ee
We recall the following result. For the reader's convenience we will give a proof of the first part of the statement  in Lemma \ref{lem:unifconv} below (in the ``global cart'', defined in Section \ref{sec:global}); the second part of the statement is \cite[Theorem 23]{BoarottoLerario}. For more details we refer the reader to \cite{AgrachevBarilariBoscain}. 

\begin{prop}\label{prop:weak}If $\gamma_n\to \gamma$ in $(\Omega^{1,2})^{\textrm{weak}}$, then $\gamma_n$ converges uniformly to $\gamma$. In particular, the map $F:\Omega\to M\times M$ is continuous (smooth indeed) for the strong topology and continuous for the weak topology. Moreover the map $\Omega\to \mathrm{hom}(L^2(I, \R^k)\times \R^m, \R^{m} \times \R^m)$ defined by $\gamma\mapsto d_\gamma F$ is continuous when $\Omega$ is endowed with the weak topology and $\textrm{hom}(L^2(I, \R^k)\times \R^m, \R^{m} \times \R^m)$ the strong topology; the same is true for the map $L^2(I, \R^k)\times \R^m\times \R\to \mathrm{hom}(L^2(I, \R^k), \R^m)$ defined by $(u,x, T)\mapsto d_uF_x^T$ (weak topology on the source, strong topology on the target). 

\end{prop}

\subsection{Sub-riemannian structures and the Energy of a path} A sub-riemannian structure on $(M,\Delta)$ is a riemannian metric on $\Delta$, i.e. a scalar product on $\Delta$ which smoothly depends on the base point. If $\Delta$ is endowed with a sub-riemannian structure, we can define the Energy:
\be J:\Omega \to \R, \quad J(\gamma)=\frac{1}{2}\int_{I}|\dot \gamma(t)|^2 dt.\ee
The Energy is a smooth map on $\Omega$, but it is only lower semincontinous on $(\Omega^{1,2})^\textrm{weak}$.  If in the above construction $\{X_i\}_{i=1,\ldots,k}$ was chosen orthonormal, in local coordinates we have $J(\gamma)=\frac{1}{2}\|u\|^2$ and  its differential is given by $d_{(u,x)}J=\langle u, \cdot \rangle$, where of course $\| \cdot\|$ denotes the $L^2$-norm and $\langle \cdot, \cdot \rangle$ the $L^2$-scalar product; by a slight abuse of notation (identifying a Hilbert space with its dual) we will sometimes simply write $d_{(u,x)}J=(u,0)\in L^{2}(I, \R^k)\times \R^m$.

\subsection{The global ``chart'' and the minimal control}\label{sec:global}Assume $M$ is a compact manifold. Given a sub-riemannian structure on $\Delta\subset TM$, there exists a family of vector fields $X_1, \ldots, X_l$ with $l\geq k$ globally defined on $M$ such that:
\be \Delta_x=\textrm{span}\{X_1(x), \ldots, X_l(x)\}, \qquad \forall x\in M.\ee
Moreover the previous family of vector fields can be chosen such that for all $x\in M$ and $u\in \Delta_x$ we have \cite[Corollary 3.26]{AgrachevBarilariBoscain}:
\be\label{eq:uminglob}
|u|^2=\inf \left\{u_1^2+\cdots + u_l^2\,\bigg|\, u=\sum_{i=1}^l u_i X_i(x)\right\},
\ee
where $|\cdot|$ denotes the modulus w.r.t. the fixed sub-riemannian structure.
Denoting by $\mathcal{U}=L^2(I, \R^l) \times M$,  we define the map $A: {\mathcal U}\to \Omega$ by:
\be\label{eq:defA}
 A(u,x)=\textrm{the curve solving the Cauchy problem $\dot\gamma =\sum_{i=1}^lu_i(t)X_i(\gamma(t))$ and $\gamma(0)=x$}.
 \ee
(We use the compactness of $M$ to guarantee that the solution of the Cauchy problem is defined for all $t\in [0,1]$). 

We will consider this construction fixed once and for all,  and call it the ``global chart''. The endpoint map for this global chart will be denoted by:
\be\label{eq:globalendpoint} 
\tilde F:\mathcal{U}\to M\times M, \quad \tilde F(u,x)=F(A(u,x)).
\ee
We also set:
 \be\label{defFxu}
\tilde F_x^t:  \mathcal{U}\to M , \quad \tilde F_x^t (u,x)=F_x^t(A(u,x))=A(u,x) (t).
 \ee
 In words, $\tilde F_x^t (u,x)$ is the evaluation at time $t$ of the curve  $A(u,x)$ defined in  \eqref{eq:defA}.

The map $A$ is continuous (both for the strong and the weak topologies on $L^{2}(I, \R^k)$) and has a right inverse $\mu:\Omega \to \mathcal{U}$ defined by:
\be \mu(\gamma)=(u^*(\gamma), \gamma(0)),\ee
where $u^*(\gamma)$ is the control realizing the minimum of $\|\cdot\|^2$ on $A^{-1}(\gamma)$. This control is called the \emph{minimal control}  \cite[Section 3.1.1]{AgrachevBarilariBoscain}; its existence and uniqueness  follows from next lemma.

\begin{lemma}\label{lem:AConvex}
The set $A^{-1}(\gamma)=\{u\in L^2(I, \R^l)\,|\, \dot \gamma=\sum_i u_i X_i\}\times \{\gamma(0)\}$ is convex and closed.
\end{lemma}

\begin{proof}
Since $A$ is continuous, then $A^{-1}(\gamma)$ is closed. Moreover if $(u, \gamma(0)), (v, \gamma(0))\in A^{-1}(\gamma)$ then:
\begin{align} \dot \gamma&=\lambda  \dot \gamma+(1-\lambda)\dot \gamma=\lambda \sum_{i=1}^l u_iX_i+(1-\lambda) \sum_{i=1}^l v_iX_i\\
&= \sum_{i=1}^l (\lambda u_i+(1-\lambda)v_i)X_i,\end{align}
which means that the convex combination $\lambda u+ (1-\lambda) v$ is still an element of $A^{-1}(\gamma)$.
\end{proof}
It is also useful to give a pointwise characterization of the minimal control. This is the aim of the next lemma.

\begin{lemma}\label{lem:Pointwiseu*}
Let  $\gamma \in \Omega$ be an admissible curve and let $u^*$ be the associated  minimal control.  Then for a.e.  $t \in I$ (more precisely at every point of differentiability of $\gamma$)  $u^*(t)$ is uniquely characterized as
\be\label{eq:Pointwiseu*}
u^*(t) = \argmin \left\{ |v|_{\R^l} \,: \, v \in \R^l \text{ satisfies }  \dot{\gamma}(t)= \sum_{i=1} ^l v_i \, X_i(\gamma(t)) \right\},
\ee
where $|\cdot|_{\R^l}$ denotes the euclidean norm in $\R^l$. 
\end{lemma}

\begin{proof}
First of all, analogously to the  the proof of  Lemma  \ref{lem:AConvex}, one can show  that $\big \{ v \in \R^l \,:\,   \dot{\gamma}(t)= \sum_{i=1} ^l v_i \, X_i(\gamma(t)) \big\}$ is a closed convex subset of $\R^l$, so there exists a unique element of minimal norm. This shows the existence \& uniqueness of the minimizer in  \eqref{eq:Pointwiseu*}. For the moment let us denote with $\bar{u} (t)$ such a  minimizer; we are then left to prove that $\bar{u}(t)=u^*(t)$ for a.e. $t \in I$.  The measurability of $\bar{u}$ is proved in \cite[Lemma 3.11]{AgrachevBarilariBoscain}, moreover  by the definition of $\bar{u}$ we have:
$$|\bar{u}(t)|_{\R^l} \leq |u^*(t)|_{\R^l} \quad \text{for  a.e. } t \in [0,1],$$
which gives that $\bar{u} \in L^2(I, \R^l)$ with $\| \bar{u}\| \leq \|u^*\|$. Since by construction  $\bar{u}\in A^{-1}(\gamma)$, the definition of $u^*$ as unique element of minimal $L^2$-norm  implies that $\bar{u}=u^*$ a.e. . 
 \end{proof}
Notice that the combination  of \eqref{eq:uminglob} and Lemma \ref{lem:Pointwiseu*} yields: 
\be \label{eq:Ju*}
 J(\gamma)=\frac{1}{2}\|u^*(\gamma)\|^2,
\ee
fact which will be useful in the next proposition.

\begin{prop}\label{prop:mu}
The map $\mu: \Omega\to {\mathcal U}= L^2(I, \R^l) \times M$ is strong-strong continuous.
\end{prop}

\begin{proof}
Let $\{\gamma_n\}_{n \in \N} \subset \Omega$ be strongly converging to $\gamma \in \Omega$ and let $u^*_n, u^*_\gamma \in L^2(I, \R^l)$ be the associated minimal controls. First of all it is clear that $\gamma_n(0)\to \gamma(0)$, so we have to prove that $u^*_n \to  u^*_\gamma$ strongly in  $L^2(I, \R^l)$.
\\ On the one hand, the strong convergence combined with \eqref{eq:Ju*}  implies that:
\be\label{u*un*}
\frac 1 2 \|u^*_\gamma\|^2 = J (\gamma)= \lim_{n\to \infty} J (\gamma_n) =\frac 1 2  \lim_{n\to \infty} \|u^*_n\|^2.
\ee
On the other hand, since the sequence $u_n^*$ is bounded in $L^2$, there exists $u_\infty \in L^2(I,\R^l)$ such that $u_n^*$ weakly converges to $u_\infty$ in $L^2(I,\R^l)$.  By the weak-strong continuity of the end point map stated in Proposition \ref{prop:weak} we get that  $A(u_\infty)=\gamma$, or in other terms  $u_\infty \in A^{-1}(\gamma)$. The definition of $u^*_\gamma$ then implies that $\|u^*_\gamma\|\leq  \|u_\infty\|$ with equality if and only if $u^*_\gamma=u_\infty$ a.e. .
\\But  the lower semicontinuity of the $L^2$-norm under weak convergence gives:
$$\|u_\infty\|\leq \liminf_{n \to \infty} \|u_n^*\| = \lim_{n \to \infty} \|u_n^*\|= \|u^*_\gamma\|, $$
where we used \eqref{u*un*}. Therefore $u_\infty=u^*_\gamma$ and $u^*_n \to u^*_\gamma$ strongly in $L^2(I, \R^l)$ as desired.
\end{proof}

\subsection{The loop space}\label{sec:G}
We say that a horizontal curve $\gamma:I \to M$ is \emph{closed} if $\gamma(0)=\gamma(1)$,  in this case $\gamma$ is also called \emph{loop}.
The horizontal basepoint free \emph{loop space} is defined by:
\be \Lambda=F^{-1}(\textrm{diagonal in $M\times M$}).\ee
In local coordinates $U\simeq L^{2}(I, \R^k)\times \R^m$ we have:
\be \Lambda \cap U=\{(u,x)\,|\, F_x^1(u)-x=0\},\ee
where $F_x^t$ was defined in \eqref{defFxu}. We also set:
$$\tilde{\Lambda}=\tilde{F}^{-1} (\textrm{diagonal in $M\times M$}) \subset {\mathcal U},$$
the counterpart of $\Lambda$ in the global chart $ {\mathcal U}=L^2(I, \R^l)\times M$. 
\\The loop space is a closed subset of $\Omega$ both for the strong and the weak topology, but in general it is not a submanifold (the endpoint map may not be a submersion). We say that a loop $\gamma$ is \emph{regular} if: 
\be\label{def:G}
 G(u,x)=F_x^1(u)-x\quad \textrm{ is a submersion at $(u_\gamma, \gamma(0))$},
 \ee
i.e. if $d_{(u_\gamma,\gamma(0))}G:  L^{2}(I, \R^k)\times \R^m \to \R^m$ is surjective.  Note this is in contrast with the  riemannian case, where
 the loop space is a smooth submanifold.   However at least in the contact case we can characterize the set $\textrm{Sing}(\Lambda)$ of its singular points: it coincides with the set of constant curves (see Proposition \ref{prop:singularities} below).
 
\begin{remark}[On the various definitions of endpoint maps]\label{remark:endpoint}It is worth at this point to collect the notations used so far for endpoint maps. On the path space $F:\Omega\to M\times M$ (no coordinates have been fixed) the endpoint map takes a curve $\gamma$ and gives its endpoints $(\gamma(0), \gamma(1))$. We will also use $F^{1}:\Omega\to M$ to denote the map giving only the final point.
Fixed a coordinate chart $U\simeq L^2\times \R^m$, the map $F(u,x)$ gives the final point of the solution of the Cauchy problem \eqref{eq:control}; similarly, fixed just a (locally) trivializing frame for $\Delta$ and charts on an open subset of  $L^2(I, \R^k)$, $F_x^t(u)$ gives the value of the solution at time $t$ of the Cauchy problem \eqref{eq:control}. In the global chart $\mathcal{U}$ the analogous maps $\tilde F$ and $\tilde F_x^t$ are defined by respectively \eqref{eq:globalendpoint} and \eqref{defFxu}. 
Finally notice that all these definitions make sense for $u\in L^2([0, T], \R^k)$ as long as $t\leq T$ (we will be using this observation in the section on the Hurewicz property).
\end{remark}

	\subsection{The Hurewicz fibration property}\label{sec:hurewicz}
Recall that a map $F:X\to Y$ is a Hurewicz fibration \cite{Hurewicz} if it has the homotopy lifting property for every space $Z$: for every homotopy $H:Z\times I\to Y$ and every lift $\tilde H_0:Z\to X$ (i.e.  every map $\tilde H_0:Z\to X$ satisfying  $F(\tilde H_0(\cdot))=H(\cdot,0)$),  there exists a homotopy $\tilde H:Z\times I\to X$ lifting $H$, i.e. $F(\tilde H(z,t))=H(z,t)$ for all $(z,t)\in Z\times I$, with $\tilde{H}(\cdot, 0)=\tilde{H}_0(\cdot)$.

In order to prove that the restricted endpoint map $F|_{\Lambda}:\Lambda \to M$  is a Hurewicz fibration, we will need two technical results from \cite{BoarottoLerario} (strictly speaking, the image of the map $F$ is contained in the diagonal in $M\times M$, but we can identify it with $M$ itself).

 \begin{prop}[\cite{BoarottoLerario}, Proposition 2]\label{propo:cross}Every point in $M$ has a neighborhood $W$ and a continuous map:
\begin{align}\hat{\sigma}: W\times W&\to L^{2}([0, \infty), \R^l)\times \R \\
(x,y)&\mapsto (\sigma(x,y), T(x,y)) \end{align}
such that  
$F^{T(x,y)}_x(\sigma(x,y))=y$ and $\hat{\sigma}(x,x)=(0,0)$ for every $x,y\in W$.\end{prop}

\begin{prop}[\cite{BoarottoLerario}, Proposition 3] \label{propo:concatenation}The map $\mathcal C: L^2(I)\times L^2([0,+\infty))\times \R\to L^2(I)$ defined below is continuous:
		\[
		\mathcal C(u,v,T)(t)=\left\{
		\begin{alignedat}{9}
		& (T+1)u(t(T+1))\quad&& 0\leq t< \frac{1}{T+1}\\
		& (T+1)v((T+1)t-1),\quad&& \frac{1}{T+1}< t\leq 1.
		\end{alignedat}
		\right.
		\]	
	Moreover (extending the definition componentwise to controls with value in $\R^l$) we also have $F_x^{1+T}(u*v)=F_x^1(\mathcal{C}(u,v,T))$ for every $x\in M$ (here $u*v$ denotes the usual concatenation).
		 \end{prop}

We will also need the following variation of the map $\mathcal{C}$ defined above:
\be \mathcal{C}^-:L^2([0, \infty))\times L^2(I)\times \R\to L^2(I).\ee
In order to define $\mathcal{C}^-$ we first define for a control $u\in L^2([0, T], \R^l)$ the backward control $B(u)(t)=u(T-t)$; notice that if $F^T_x(u)=y$, then $F_y^T(B(u))=x.$
We define then:
\be \mathcal{C}^-(u, v, T)=B\left(\mathcal{C}(B(v), B(u|_{[0, T]}), T)\right).\ee
Since $B:L^2\to L^2$ is continuous, then Proposition \ref{propo:concatenation} implies that $\mathcal{C}^-$ is continuous as well. Essentially the flow associated to the control $\mathcal{C}^-$ first  goes through  $u|_{[0,T]}$ and then goes through $v$; note that  if $T=0$ it just reduces to the flow of $v$. 
In the next lemma we prove that $\tilde{F}|_{\tilde{\Lambda}}:\tilde \Lambda\to M$ is a Hurewicz fibration (as above, the image of $\tilde{F}$ is in the diagonal in $M\times M$, which we identify with $M$ itself); this will be the key technical property in order to investigate the topological structure of $\Lambda$. After this technical step, in Theorem \ref{thm:hurewicz} we will prove that $F|_\Lambda$ itself is a Hurewciz fibration. Recall that $\tilde{\Lambda}:=\tilde{F}^{-1}(M)\subset L^{2}(I, \R^{l})\times M$.

\begin{lemma}
The map $\tilde F|_{\tilde \Lambda}:\tilde\Lambda\to M$ is a Hurewicz fibration. \end{lemma}

\begin{proof}By Hurewicz Uniformization Theorem \cite{Hurewicz}, it is enough to show that the homotopy lifting property holds locally, i.e. every point $x\in M$ has a neighborhood $W$ such that $\tilde{F}|_{\tilde{F}^{-1}(W)}$ has the homotopy lifting property with respect to any space.

Pick $x\in M$ and let $W$ be the neighborhood given by Proposition \ref{propo:cross}, together with the corresponding $\hat{\sigma}=(\sigma, T)$.
Consider a homotopy $H:Z\times I\to W$ (here $Z$ is any topological space) and a lift $\tilde H_0:Z\to \tilde \Lambda$ (i.e. $\tilde F(\tilde H_0(z))=H(z,0)$ for all $z\in Z$).
\\
Setting $T(z,t)=T(H(z,0), H(z,t))$, and denoting by $p_1:L^2(I, \R^l)\times M\to L^2(I, \R^l)$ the projection on the first factor,  we define the lifting homotopy $\tilde{H} :Z\times I\to\mathcal{U}$ by:
\be \tilde{H}(z,t)=\bigg(\mathcal{C}^-\left( \sigma(H(z,t), H(z,0)), \mathcal{C}\left(p_1(\tilde{H}_0(z)), \sigma(H(z,0), H(z,t)), T(z,t)\right), T(z,t)\right),H(z,t)\bigg).\ee
By construction $\tilde{H}$ is continuous, it lifts $H$ and $\tilde{H}(Z\times I)\subset \tilde \Lambda.$  This proves that  $\tilde{F}|_{\tilde{F}^{-1}(W)}$ has the homotopy lifting property with respect to any space, and since we can cover $M$ with neighborhoods of the form $W$ as above the result follows.
\end{proof}

\begin{thm}\label{thm:hurewicz}
The endpoint map $F|_\Lambda:\Lambda \to M$ is a Hurewicz fibration.
\end{thm}
\begin{proof}
Let $H:Z\times I\to M$ be a homotopy and $\tilde{H}_0:Z\to \Lambda$ be a lift of $H_0$. Consider the map:
\be \hat H_0:Z\to \mathcal{U}, \quad \hat{H}_0(z)=\mu(\tilde{H}_0(z)). \ee
Since $\mu$ is a continuous right inverse for $A$ (by Proposition \ref{prop:mu}), then $\tilde F (\hat H_0(z))=F(A(\mu( \tilde{H}_0(z)))=F(\tilde{H}_0(z))=H(z,0)$ and $\hat{H}_0$ is a lift of $H_0$. By the previous lemma there exists a lifting homotopy $\tilde H:Z\times I\to \tilde \Lambda$ for $H:Z\times I\to M.$ The map:
\be \hat H:Z\times I\to \Lambda, \quad \hat H(z,t)=A(\tilde H(z,t))\ee
is a homotopy lifting $H$.
\end{proof}

	\subsection{Homotopy type of the loop space}
	As a corollary of Theorem \ref{thm:hurewicz} we can derive the following result, which is an analogue of the classical one in riemannian geometry (we refer the reader to \cite{Oancea} for a survey of the classical results). 
\begin{thm}\label{thm:homotopy}For every $k\geq 0$ we have the following isomorphism\footnote{See Remark \ref{rmk:semi} on the semidirect product.}:
\be\label{eq:ho} \pi_{k}(\Lambda)\simeq \pi_k(M) \ltimes \pi_{k+1}(M).\ee
\end{thm}

\begin{proof}
Since $F|_{\Lambda}:\Lambda \to M$ is a Hurewciz fibration, then we have a long exact sequence \cite{Spanier}:
\be\label{eq:long} \cdots\rightarrow \pi_{k}(\Omega_{x,x})\rightarrow \pi_k(\Lambda)\stackrel{F_*}{\longrightarrow} \pi_k(M)\rightarrow \pi_{k-1}(\Omega_{x,x})\rightarrow \cdots\ee
where $\Omega_{x,x}$ denotes the set of admissible curves starting and  ending at a fixed point $x$. Recall that \cite[Section 2.3]{BoarottoLerario}:
\be\pi_{k}(\Omega_{x,x})\simeq \pi_{k+1}(M)\quad \forall k\geq 0.\ee
Also notice that the map $s:M\to \Lambda$ defined by:
\be \label{eq:defs}
s(x)=\textrm{constant curve $\gamma_x$ such that $\gamma_x(t)\equiv x$}
\ee
is continuous and defines a section of $F$ (i.e. it is a right inverse of $F$). In particular $F_*$ is surjective and the sequence in \eqref{eq:long} splits as:
\be \label{eq:split}0\rightarrow \pi_{k+1}(M)\to \pi_{k}(\Lambda)\stackrel{F_*}{\longrightarrow} \pi_{k}(M)\rightarrow 0\ee
and the result follows.
\end{proof}

The homotopy of $\Lambda$ can be compared with the homotopy of the ``standard'' (non horizontal) loop space $\Lambda_\textrm{std}$, endowed with the $W^{1,2}$ topology. In fact we have the following much stronger result, asserting that the inclusion of one space into the other is a strong homotopy equivalence.

\begin{thm}\label{thm:i*iso}The horizontal loop space has the homotopy type of a CW-complex and the inclusion $i:\Lambda\hookrightarrow \Lambda_{\textrm{std}}$ is a strong homotopy equivalence. 
\end{thm}

\begin{proof}Since the map $F|_{\Lambda}:\Lambda\to M$ is a Hurewicz fibration and the base $M$ has the homotopy type of a CW-complex as well as any fibre $F|_{\Lambda}^{-1}(x)=\Omega_{x,x}$ (by \cite[Theorem 5]{BoarottoLerario}), then the total space $\Lambda$ has the homotopy type of a CW-complex by \cite[Proposition 5.4.2]{rure}.

Since $\Lambda_{\textrm{std}}$ has the homotopy type of a CW-complex, to prove that the inclusion $i:\Lambda\hookrightarrow \Lambda_{\textrm{std}}$ is a strong homotopy equivalence it is enough to prove that the map: $i_*: \pi_k(\Lambda) \to \pi_k( \Lambda_{\textrm{std}})$ is an isomorphism; the result will then follow from Milnor's extension of Withehead's theorem \cite[Lemma 5.1]{MilnorCW} (see also \cite[Lemma 5.3.2]{rure}).

The fact that $i_*$ is an isomorphism immediately follows from the naturality of the long exact sequences of Hurewicz fibrations and the fact that the splitting in \eqref{eq:long} yields for every $k\geq 0$ a commutative diagram:
\be\label{eq:diagram}
\begin{tikzcd}
0\arrow{r}&\pi_{k+1}(M)\arrow{r}\arrow{d}{\wr}
&\pi_{k}(\Lambda)\arrow{r}\arrow{d}{i_*}
&\pi_k(M)\arrow{d}{\wr}\arrow{r}&0\\
0\arrow{r}&\pi_{k+1}(M)\arrow{r}&\pi_{k}(\Lambda_{\textrm{std}})\arrow{r}&\pi_{k}(M)\arrow{r}&0
\end{tikzcd}
\ee
Since the two extremal vertical arrows are isomorphisms, then $i_*$ is also an isomorphism.
\end{proof}

\begin{remark}\label{rmk:semi}Notice that in the statement of Theorem \ref{thm:homotopy} the group $\pi_{k}(\Lambda)$ is abelian for $k\geq 2$, hence the action of $\pi_k(M)$ on $\pi_{k+1}(M)$ in \eqref{eq:ho} is trivial; the \emph{semidirect} product description is especially interesting only in the case $k=1$, for which the action of $\pi_1(M)$ on $\pi_2(M)$ is the same one as resulting from the short exact sequence:
\be0\rightarrow \pi_{k+1}(M)\to \pi_{k}(\Lambda_{\textrm{std}})\stackrel{F_*}{\longrightarrow} \pi_{k}(M)\rightarrow 0.\ee
This last statement follows from the fact that the diagram \eqref{eq:diagram} is commutative.
\end{remark}
The following corollary sharpens the local structure of $\Lambda$ near a singular loop.
\begin{cor}\label{cor:contractible}Every loop $\gamma\in \Lambda$ (in particular a singular point of $\Lambda$) has a neighborhood $U$ such that the inclusion $U\hookrightarrow\Lambda$ is homotopic to a trivial loop (i.e. a constant).
\end{cor}
\begin{proof}
Since $\Lambda$ has the homotopy type of a CW-complex by Theorem \ref{thm:i*iso} above, then the result follows from \cite[Proposition 5.1.2]{rure}.
\end{proof}

\subsection{Deformation of homotopy classes with small energy}\label{ss:deformHom}
\begin{thm}\label{thm:epsilon}Let $M$ be a compact sub-riemannian manifold. There exists $\epsilon>0$ such that if $f:S^k\to \Lambda$ is a continuous function satisfying:
\be \sup_{\theta \in S^k}J(f(\theta))\leq \epsilon,\ee
then $f$ is homotopic to a map $f'$ with values in the set of constant curves $s(M)\subset \Lambda$:
\be f\sim f', \quad \text{for some } f':S^k\to s(M),\ee
where the map  $s:M\to \Lambda$ was defined in \eqref{eq:defs}. More precisely one can choose $f'(\theta)=s(f(\theta)(0))$.
\end{thm}

\begin{proof}
Let us consider a splitting $TM=\Delta\oplus \Delta^\perp$ and a riemannian metric $g$ on $TM$  such that:
\be g|_{\Delta}=\textrm{the given sub-riemannian metric}.\ee
Let us denote by $B(x, \epsilon)$ the riemannian ball centered at $x$ of radius $\epsilon$ and by $B_\textrm{sr}(x, \epsilon)$ the sub-riemannian one. Since $M$ is compact there exists $\epsilon$ such that for every $x\in M$ the ball $B(x, \sqrt{2\epsilon})$ is geodesically convex (with respect to $g$). Moreover since $d_{g}(x,y)\leq d_{\textrm{sr}}(x,y)$ we can also assume that for such an $\epsilon$ and for all $x\in M$:
\be B_{\textrm{sr}}(x, \sqrt{2\epsilon})\subseteq B(x, \sqrt{2\epsilon}). \ee
Now the hypothesis that $J(f(\theta))\leq \epsilon$ implies that the sub-riemannian length of the path $f(\theta)$ is smaller than $\sqrt{2\epsilon}$ and consequently:
\be f(\theta)(t)\in  B_{\textrm{sr}}(f(\theta)(0), \sqrt{2\epsilon})\subseteq B(f(\theta)(0), \sqrt{2\epsilon})\quad \forall t\in [0,1].\ee
Define now the map:
\be H(\theta)(s,t)=\gamma_{t, \theta}(s)\ee
where $\gamma_{t, \theta}(\cdot):[0,1]\to M$ is the unique riemannian length minimizing geodesic from $f(\theta)(t)$ to $f(\theta)(0)$. Because of its uniqueness, $\gamma_{t, \theta}$ depends continuously in $t$ and $\theta$ and
\be H: S^k\times I\to \Lambda_{\textrm{std}}, \quad (\theta, s)\mapsto  \gamma_{\cdot,\theta}(s)\ee
defines a homotopy in $\Lambda_{\textrm{std}}$ bewtween $H(\cdot, 0)=f$ and the map $f'=H(\cdot, 1)$ defined by:
\be f'(\theta)=s(f(\theta)(0)).\ee
In other words, denoting by $h:S^k\to M$ the map $h(\theta)=f(\theta)(0)$, we have:
\be \label{eq:id}[f]_{\pi_{k}(\Lambda_{\textrm{std}})}=[f']_{\pi_{k}(\Lambda_{\textrm{std}})}=[s\circ h]_{\pi_{k}(\Lambda_{\textrm{std}})}=s_*[h]_{\pi_{k}(M)}.\ee
By the naturality and commutativity of the right square in \eqref{eq:diagram} we have the following commutative diagram (vertical arrows are isomorphisms):
\be
\begin{tikzcd}
\pi_{k}(\Lambda)\arrow{r}\arrow{d}{i_*}
&\pi_{k}(M)\arrow{d}\arrow[bend left]{l}{s_*}\\
\pi_{k}(\Lambda_{\textrm{std}})\arrow{r}&
\pi_{k}(M)\arrow[bend left]{l}{s_*}
\end{tikzcd}
\ee
The commutativity of this diagram, together with $i_*([f]_{\pi_{k}(\Lambda)})=[f]_{\pi_{k}(\Lambda_{\textrm{std}})}=s_*[h]_{\pi_k(M)}$ (this is the content of the chain of equalities in \eqref{eq:id}) and Theorem \ref{thm:i*iso}  finally give:
\be [f]_{\pi_{k}(\Lambda)}=s_*[h]_{\pi_k(M)}=[f' ]_{\pi_k(\Lambda)}.\ee
\end{proof}

\begin{cor}\label{cor:nonzero}Let $M$ be a compact sub-riemannian manifold and let $k\in \mathbb{N}$ be  such that $\pi_{k+1}(M)\neq 0$ and $\pi_k(M)=0$. Then there exists $\epsilon>0$ such that for every \emph{nonzero} $\alpha\in \pi_{k}(\Lambda)$:
\be \inf_{[f]=\alpha}\,\sup_{\theta \in S^k}J(f(\theta))\geq \epsilon,\ee
where the $\inf$ is taken over all continuous maps $f:S^k \to \Lambda$ such that $[f]_{\pi_k(\Lambda)}=\alpha$.
\end{cor}
\begin{proof}
First of all, thanks to Theorem \ref{thm:homotopy} we have that $\pi_{k}(\Lambda)\simeq \pi_{k+1}(M)\neq 0$.
Assume that for some nonzero $\alpha\in \pi_{k}(\Lambda)$ we have $\inf_{[f]=\alpha}\,\sup_{\theta \in S^k}J(f(\theta))=0$.
Let $\epsilon>0$ be given by Theorem \ref{thm:epsilon} and consider $f:S^{k}\to \Lambda$ representing $\alpha$ such that:
\be\sup_{\theta \in S^k}J(f(\theta))<\epsilon.\ee
 Then, by Theorem \ref{thm:epsilon}, $f\sim f'$ with $f'=s\circ h$ and $h:S^{k}\to M$. Thus:
 \be \alpha=[f]_{\pi_k(\Lambda)}=[f']_{\pi_{k}(\Lambda)}=s_*[h]_{\pi_k(M)}.\ee
 On the other hand, by  assumption on $k$, we have that  $\pi_{k}(M)$ is zero and consequently $[h]=0$, contradicting $\alpha \neq 0.$
 \end{proof}
 
 \subsection{Closed sub-riemannian geodesics}\label{sec:geo}
	Let us start with the definition of closed sub-riemannian geodesic (to be  compared with \cite[Section 4.1]{BoarottoLerario}).
	
	\begin{defi}[Closed sub-riemannian geodesic]\label{def:closedGeod}
	A non-constant curve $\gamma:S^1\to M$ is called a \emph{closed sub-riemannian geodesic} if it satisfies the following properties: 
	\begin{enumerate}
	\item it is absolutely continuous,
	\item  its derivative (which exists almost everywhere) belongs   to the sub-riemannian distribution,
	\item it is parametrized by constant speed,
	\item  it is locally length minimizing, in the sense that for every $\theta\in S^1$ there exists $\delta(\theta)>0$ such that $\gamma|_{[\theta-\delta(\theta), \theta+\delta(\theta)]}$ is length minimizing, i.e.    the restriction $\gamma_{[\theta-\delta(\theta), \theta+\delta(\theta)]}$ has minimal length among all horizontal curves joining $\gamma(\theta-\delta(\theta))$ with $\gamma(\theta+\delta(\theta))$.	\end{enumerate}
	\end{defi}
	
The following proposition gives a sufficient condition for a regular curve $\gamma \in \Lambda$ (in the sense of \eqref{def:G}) to be a closed geodesic.  Let us mention that \eqref{eq:lagrange} below corresponds to the (normal)
Lagrange multiplier rule associated to the problem of  extremizing (i.e. finding a minimum or more generally a critical point of)
the Energy functional $J$ among all loops (with no fixed base point). Clearly the global
minimum is $0$ but here, in the spirit of Morse theory, one is interested on more general critical points (typically local minima or  of saddle type), which are non-trivial.
\\Notice also that  \eqref{eq:lagrange} is equivalent to $\nabla (J|_{\Lambda})=0$ at $(u,x)$.

\begin{prop}\label{prop:CharClosedG}
Let $M$ be a sub-riemannian manifold and $\gamma:I\to M$ be a closed horizontal curve such that there exists  a nonzero $\lambda \in \R^m$ with the property that in some local coordinates around $\gamma=\gamma_{(u,x)}$ we have:
\be  \label{eq:lagrange}
\lambda d_{(u,x)}G=d_{(u,x)}J,
\ee
where $G(u,x)=F_x^1(u)-x$.
Then $\gamma$ is the projection of a periodic trajectory for the sub-riemannian Hamiltonian vector field (i.e. it is the projection of a periodic  sub-riemannian normal extremal); in particular $\gamma$ is smooth and, identifying the endpoints of the interval $[0,1]$, it extends to a closed sub-riemannian geodesic. \end{prop}

\begin{proof}
Given coordinates on a neighborhood $U\simeq L^{2}(I, \R^k)\times \R^m$ of $\gamma \in \Omega$, we can write:
\be \Lambda \cap U=\{(u,x)\,|\, G(u,x)=F_x^1(u)-x=0\}.\ee

Let us denote by $\varphi_{u}^{s,t}:W\to M$ the flow from time $s$ to time $t$ of the (time dependent) vector field $\sum_i u_iX_i$; then we can write:
\be d_{(u,x)}G=\left (d_uF_x^1, d_x\varphi_{u}^{0,1}-\mathbbm{1}\right)\quad \textrm{and}\quad d_{(u,x)}J=(u, 0).\ee
Thus equation \eqref{eq:lagrange} can be rewritten as:
\be \label{eq:lagrange2}\lambda d_uF_x^1=u\quad \textrm{and}\quad \lambda(d_x\varphi_{u}^{0,1}-\mathbbm{1})=0.\ee
The equation on the left of \eqref{eq:lagrange2} says that $u$ is the projection of the normal extremal \cite[Proposition 8.9]{AgrachevBarilariBoscain}:
\be \lambda:I\to T^*M, \quad \lambda(t)=(\varphi_u^{t,1})^*\lambda.\ee
The equation on the right of \eqref{eq:lagrange2} is equivalent to:
\be\lambda(0)=(\varphi_u^{0,1})^*\lambda=\lambda=\lambda(1)\ee
which tells exactly that the extremal $\lambda(t)$ is periodic.

On the other hand \cite[Theorem 4.61]{AgrachevBarilariBoscain} (see also \cite[Proposition 14]{BoarottoLerario}) implies that the curve $\gamma$ is locally length minimizing and parametrized by constant speed. Since it is the projection of a periodic trajectory of the sub-riemannian Hamiltonian field, identifying the endpoints of the interval of definition gives a map $\gamma:S^1\to M$ satisfying properties (1)-(4) above.
\end{proof}

\subsection{Closed geodesics realizing a given class in $\pi_1(M)$}
Before going to the core of the paper, which will consist in the min-max construction of a closed geodesic in simply connected  contact manifolds, in this short section we show that if  $\pi_1(M)$ is not trivial then there always exists a closed geodesic, no matter the sub-riemannian structure is contact or not. This will be achieved by a minimization process, well known in the literature  as ``direct method in the calculus of variations''.
 
\begin{thm}\label{thm:pi1}
Let $M$ be a compact, connected sub-riemannian manifold such that $\pi_{1}(M)\neq 0$. Then for every \emph{nonzero} $\alpha\in \pi_1(M)$ there exists a closed sub-riemannian geodesic $\bar{\gamma}_\alpha:S^1\to M$ such that $[\bar{\gamma}_\alpha]_{\pi_1(M)}=\alpha.$ Moreover $\bar{\gamma}_\alpha$ minimizes the energy $J$ (and thus also the sub-riemannian length) in its homotopy class:
$$J(\bar{\gamma}_\alpha)=\min \left\{ J(\gamma) \,: \, \gamma \in \alpha \right\}.$$
\end{thm}
\begin{proof}
First of all, since by assumption $M$ is connected, we have $\pi_0(M)=0$. By Theorem \ref{thm:homotopy} it follows  that $\pi_0(\Lambda)\simeq \pi_1(M)$, and   the two will be identified in the rest of the proof. 
Moreover, the assumption that $\alpha\in \pi_0(\Lambda)$ is non null implies, thanks to Corollary \ref{cor:nonzero}, that there exists $\epsilon=\epsilon(\alpha)>0$ such that:
\be\label{eq:infJalpha}
\inf_{\gamma \in \alpha}   J(\alpha) = \epsilon_\alpha >0.
\ee
Let $\gamma_n \in \alpha$ be a minimizing sequence for $J$ (i.e.  $J(\gamma_n)\to \epsilon_\alpha$ as $n\to \infty$) and call $(u_n^*,x_n=\gamma_n(0)) \in {\mathcal U}$ the associated  minimal controls  and the initial points respectively.  By the relation \eqref{eq:Ju*} we clearly have that   
\be\label{eq:Junconv}
\frac 1 2 \|u_n^*\|^2 =   J(\gamma_n) \to   \epsilon_\alpha \quad \text{as } n\to \infty.
\ee 
In particular the sequence  $\{u_n^*\}_{n \in \N} \subset L^2(I, \R^l)$ is bounded and thus weakly converges, up to subsequences, to some control  $\bar{u} \in L^2(I, \R^l)$. On the other hand, the compactness of $M$ ensures that, again up to subsequences,  there exists $\bar{x}$    such that $x_n \to \bar{x}$.  Called $\bar{\gamma}=A(\bar{u},\bar{x})$ the associated limit curve as in \eqref{eq:defA}, the lower-semicontinuity of the norm under weak convergence gives
\be
J(\bar{\gamma}) \leq \frac 1 2  \| \bar{u} \|^2 \leq  \frac 1 2 \liminf_{n \to \infty} \| u^*_n\|^2 = \liminf_{n \to \infty}   J(\gamma_n)= \epsilon_\alpha.
\ee
Thus,  in order to conclude the proof,  it is enough to show that $\bar{\gamma} \in \alpha$.
\\To this aim, using that the  next Lemma \ref{lem:unifconv} yields the uniform convergence of $\gamma_n$ to $\bar{\gamma}$, we infer that: 
$$i_* (\alpha)=i_*([\gamma_n]_{\pi_0(\Lambda)})=  [ \gamma_n ]_{\pi_0(\Lambda_{\textrm{std}})}=  [\bar{\gamma}]_{\pi_0(\Lambda_{\textrm{std}})}=   i_*([\bar{\gamma}]_{\pi_0(\Lambda)}).$$
But now $i_*: \pi_0(\Lambda) \to  \pi_0(\Lambda_{\textrm{std}})$ is an isomorphism thanks to Theorem \ref{thm:i*iso}, therefore  we conclude that  $[\bar{\gamma}]_{\pi_0(\Lambda)}=\alpha$,  as desired.
\\Since $\bar{\gamma}$ minimizes the functional $J$ in its homotopy class, we have that $\nabla (J|_{\Lambda})=0$ at $\bar{\gamma}$  and thus $\bar{\gamma}$ is a smooth closed sub-riemannian geodesic in virtue of Proposition \ref{prop:CharClosedG}.
\end{proof}

In the proof of Theorem  \ref{thm:pi1} we have used the following result, well known to experts but whose proof we recall for the reader's convenience.
\begin{lemma}\label{lem:unifconv}
Let $(u_n,x_n), (\bar{u},\bar{x}) \in {\mathcal U}=L^2(I,\R^l)\times M$ be pointed controls in the global chart and let $\gamma_n=A(u_n,x_n), \bar{\gamma}=A(\bar{u},\bar{x})$ be the associated curves as in \eqref{eq:defA}. Assume that 
$$x_n=\gamma_n(0) \to \bar{\gamma}(0)=\bar{x} \quad \text{in }M, \qquad \text{and} \qquad u_n  \weakto \bar{u} \quad \text{weakly in }L^2(I,\R^l).$$
Then $\gamma_n \to \bar{\gamma}$ uniformly.
  \end{lemma}

\begin{proof}
It is enough to show  that there exists $T=T(\sup_{i=1,\ldots,l} \|X_i\|_{C^1(M)}, \sup_{n \in \N} \|u_n\|)>0$ such that $\gamma_n|_{[0,T]}\to \bar{\gamma}|_{[0,T]}$ uniformly.
\\Since by assumption $u_n \weakto \bar{u}$, by Banach-Steinhaus Theorem we know that $\sup_{n} \|u_n\|<\infty$ and therefore  the Cauchy-Schwartz inequality  gives
$$d_{\textrm{SR}}(\gamma_n(t), \gamma_n(0)) \leq \int_0^t |\dot \gamma_n| \leq t^{1/2} \, \left( \int_0^t |\dot \gamma_n|^2\right)^{1/2} \leq  \sqrt{2 t}\, J(\gamma_n)^{1/2} \leq  t^{1/2} \sup_n \|u_n\|,  $$
where  $d_{\textrm{SR}}$ is of course the sub-riemannian distance. In particular, as  $x_n \to \bar{x}$,  there exists  $T=T(\sup_n \|u\|_n)>0$  such that  $\gamma_n|_{[0,T]}$ are all contained in a fixed coordinate neighborhood $W$ of $\bar{x}$. Observe also that there exists $C_L>0 $ such that 
\be\label{eq:XiLip}
 | X_i (x_1) - X_i(x_2) | \leq C_L  |x_1-x_2| \quad \forall x_1,x_2 \in W.
 \ee
For $t \in  [0,T]$ we can then estimate
\begin{eqnarray}
|\bar{\gamma}(t)-\gamma_n(t)|&\leq& |\bar{\gamma}(0)-\gamma_n(0)| + \left|\int_0^t \sum_{i=1}^l \bar{u}_i(s) X_i(\bar{\gamma}(s))-u_{n,i}(s) X_i(\gamma_n(s)) \, d s   \right| \nonumber \\
&\leq& |\bar x- x_n|+ \left|\int_0^t \sum_{i=1}^l (\bar{u}_i(s)- u_{n,i}(s)) \; X_i(\bar{\gamma}(s)) \, d s  \right|  \nonumber \\
         && \qquad \qquad  +   \left| \int_0^t \sum_{i=1}^l \bar{u}_i(s) \big(X_i(\bar{\gamma}(s))- X_i(\gamma_n(s))\big) \, d s   \right| .
\end{eqnarray}
Now the integral in the first line converges to $0$ as $n\to \infty$  due to the weak convergence $u_n\weakto \bar{u}$ in $L^2(I, \R^l)$. The integral  in the second line can be easily estimated using   \eqref{eq:XiLip} and the Cauchy-Schwarz inequality as:
$$ \left| \int_0^t \sum_{i=1}^l \bar{u}_i(s) \big(X_i(\bar{\gamma}(s))- X_i(\gamma_n(s))\big) \, d s   \right| \leq C  \,  \sqrt{T} \, \|\bar u\| \, \sup_{s \in [0,T]} |\gamma_n(s)-\bar{\gamma}(s)| , $$
where $C$ depends just on $C_L$ and dimensional constants. Since by lower semicontinuity we know that  $\|\bar u\| \leq  \liminf_n \|u_n\|  \leq \sup_n \|u_n\|$, the combination of the three last inequalities gives
$$\sup_{t \in [0,T]} |\bar{\gamma}(t)-\gamma_n(t)|\leq \epsilon_n + C \sup_{n \in \N} \|u_n\|  \, \sqrt{T} \,  \sup_{t \in [0,T]} |\bar{\gamma}(t)-\gamma_n(t)|,$$
where $\epsilon_n \to 0$ as $n\to \infty$.
Choosing finally $T=T(C_L, \sup_{n \in \N} \|u_n\|)>0$ such that $ C \sup_{n \in \N} \|u_n\|  \, \sqrt{T} \leq 1/2$, we can absorb the rightmost term into the left hand side and get 
$$\sup_{t \in [0,T]} |\bar{\gamma}(t)-\gamma_n(t)|\leq  2 \epsilon_n \to 0, $$ 
as desired.
\end{proof}

\section{Contact sub-riemannian manifolds}\label{sec:contact}
A contact sub-riemannian manifold is a sub-riemannian manifold $(M, \Delta)$ such that $\Delta\subset TM$ is a contact distribution (see \cite[Section 6.1.2]{Montgomery} for more details). 
\subsection{Singularities of the loop space in the contact case}
Recall that a loop $\gamma \in \Lambda$ is said \emph{regular} (or smooth) if the following holds: denoted with $U\simeq L^{2}(I, \R^k)\times \R^m$ a neighborhood   of  $\gamma$ and with $G(u,x)=F_x^1(u)-x$, then $G$ is a submersion at $\gamma$. The set of regular points of $\Lambda$ is denoted by $\textrm{Reg}(\Lambda)$.  If instead $\gamma$ is not regular then it is called \emph{singular} and  the family of singular points of $\Lambda$ is denoted with $\textrm{Sing}(\Lambda)$.

\begin{prop}\label{prop:singularities}
If $M$ is a \emph{contact} manifold, then:
\be \textrm{Sing}(\Lambda)=\{\gamma:I\to M\,|\, \gamma(t)\equiv \gamma(0)\}.\ee
\end{prop}
\begin{proof}
Let us fix coordinates on a neighborhood $U\simeq L^{2}(I, \R^k)\times \R^m$ of a curve $\gamma$. Then we can write:
\be \Lambda \cap U=\{(u,x)\,|\, G(u,x)=F_x^1(u)-x=0\}.\ee
Denoting as above by  $\varphi_{u}^{s,t}:W\to M$ the flow from time $s$ to $t$ of the vector field $\sum_i u_iX_i$, the differential of $G$ acts as:
\be \label{eq:diffcontact}d_{(u,x)}G(\dot u, \dot x)=d_uF_x^1\dot u +d_x\varphi_{u}^{0,1}\dot x-\dot x.\ee
Recall that in the contact case the only critical point of the endpoint map $F_x^1$ is the zero control (\cite[Corollary 4.40]{AgrachevBarilariBoscain}, see also \cite[Chapter 5]{Montgomery}). In particular if $\gamma$ is not a constant curve, then $d_uF_x^1$ is a submersion and consequently $d_{(u,x)}G$ is a submersion as well, implying that $\gamma$ is a regular point of $\Lambda$.
On the other hand, if $\gamma$ is a constant curve, then $u=0$ and the differential $d_0F_x^1$ is not a submersion (its image equals $\Delta_x$); moreover $\varphi_u^{s,t}=\textrm{id}_M$ and consequently $d_x\varphi_{0}^{0,1}\dot x=\dot x$. Substituting this into \eqref{eq:diffcontact} we get:
\be \textrm{im} (d_{(0,x)}G)=\Delta_x,\ee
which shows that $\gamma$ is a singular point of $\Lambda.$
\end{proof}

	\subsection{Palais-Smale property}\label{sec:PS}
	We denote by $g$ the restriction of $J$ to $\Lambda$:
	\be g:\Lambda\to \R, \quad g(\gamma)=J(\gamma).\ee

	Notice that on $\textrm{Reg}(\Lambda)$ the gradient of $g$, denoted with $\nabla g\in L^2(I, \R^k)$, is well defined and coincides with the  projection of $dJ$ on the tangent space to $\Lambda$. On the other hand, if $\gamma \in \textrm{Sing}(\Lambda)$ and $(M, \Delta)$ is contact, then by Proposition \ref{prop:singularities} we know that $\gamma$ is a constant curve and thus $d_{\gamma}J=0$. It is then natural (and we will use this convention) to set $\nabla_{\gamma} g=0$ in this case.

	The following theorem is one of the new main technical tools introduced in this paper, and will play a crucial role in the proof of  the existence of a closed geodesic in contact sub-riemannian manifolds.  
	
	\begin{thm}[Palais-Smale property holds for contact manifolds]\label{thm:PS}
	Let $M$ be a compact, connected, contact sub-riemannian manifold and  let $\{\gamma_n\}_{n\in \N}\subset \Lambda$ be a sequence such that:
	\be g(\gamma_n)\leq E\quad \textrm{and}\quad  \|\nabla_{\gamma_n}g\|\to 0.\ee
	Then there exist  $\bar{\gamma}\in \Lambda$ and a subsequence $\{\gamma_{n_k}\}_{k\in \N}$ such that $\gamma_{n_k}\to \overline \gamma$ strongly in $\Omega^{1,2}$.
	\end{thm}
	
	\begin{proof}
	First of all if $\gamma_n$ is a constant curve for infinitely many $n$'s, then trivially  the compactness of $M$ ensures the existence of a limit constant curve $\bar{\gamma}$ such that  the thesis of the theorem holds. Therefore without loss of generality we can assume $\gamma_n \in \textrm{Reg} (\Lambda)$ for every $n \in \N$.
	
	Consider a global chart $\mathcal{U}$ and set $u^*_n=u^*(\gamma_n)$ the minimal control associated to $\gamma_n$. Since by \eqref{eq:Ju*} we know that $ \frac 1 2 \|u^*_n\|^2=g(\gamma_n)\leq E$, there exists a weakly converging subsequence of $\{u_n^*\}_{n\in \N}$ (still call it $u_n^*$); using the compactness of  $M$, we can also  assume that also the sequence of starting points converges $x_n=\gamma_n(0)\to \overline x$:
	\be (u_n^*, x_n)\rightharpoonup (\overline u, \overline x).\ee
	The weak-strong continuity of $\tilde F$ stated in Proposition \ref{prop:weak} then implies: 
	\be
	 \tilde F(\overline u, \overline x)=\lim_{n\to \infty}\tilde F(u^*_n, x_n)=\lim_{n\to \infty} (x_n,x_n)=(\bar x, \bar x),
	 \ee
	which proves that $\overline\gamma =A(\overline u, \overline x)$ is a closed curve. Now two possibilities arise: 
	\begin{enumerate}
	\item  $\overline\gamma$ is a regular point of $\Lambda$; 
	\item $\overline\gamma$ is a singular point of $\Lambda$, and in particular by Proposition \ref{prop:singularities} it is a constant curve. 
	\end{enumerate}
	We will deal with these two cases separately.

	\subsubsection{The case $\overline \gamma$ is a regular point of $\Lambda$}
	 In this case the proof proceeds in a similar way as in \cite[Proposition 10]{BoarottoLerario}; we sketch it here for the reader's convenience. 	 
	 Let $U\simeq L^2(I, \R^{2a})\times \R^{2a+1}$ be a chart containing $\overline \gamma=(\overline u, \overline x)$ (here $k=2a$ and $m=2a+1$ since we are in the contact case); in this chart:
	 \be \dot {\overline \gamma}(t)=\sum_{i=1}^{2a}\overline u_i(t)X_i(\overline\gamma (t), t),\quad \gamma(0)=\overline x.\ee
	 Since $\gamma_{n}$ converges to $\overline \gamma$ uniformly, then eventually there exists a sequence of associated controls $\{(u_n, x_n)\}_{n\geq \overline n}$ in this chart such that $\gamma_n=\gamma_{(u_n, x_n)}$ for all $n\geq \overline n.$ 
	 
	 After possibly shrinking $U$, we have that $\Lambda\cap U=G^{-1}(0)$, where $G$ was defined by $G(u,x)=F_x^1(u)-x$ and the equation $G=0$ is regular in $U$ (because $\overline{\gamma}$ is a regular point of $\Lambda$, the differential of $G$ is therefore submersive on a neighborhood of $(\overline u, \overline x)$), see Section \ref{sec:G}.
	 
	  Let $\{e_1, \ldots, e_{2a+1}\}$ be a fixed basis for $\R^{2a+1}$ and for every $n\in \N$ and $i=1, \ldots, 2a+1$ define:
	 \be \label{eq:defwpullback}
	 w_i^n=d_{(u_n, x_n)}G^* e_i\quad \textrm{and}\quad  W^n=\textrm{span}\{w_1^n, \ldots, w_{2a+1}^n\},
	 \ee
	 where $d_{(u_n, x_n)}G^*:\R^m \to L^2(I, \R^{2a})\times \R^{2a+1}$ is the adjoint operator (here we identify a Hilbert space with its dual). 
	The condition that the point $(\bar{u}, \bar{x})$ is a regular point of $\Lambda$ ensures that $d_{(\bar{u},\bar{x})}G$ is a submersion, hence $d_{(\bar{u}, \bar{x})}G^*$ has maximal rank (i.e. it is injective).
	
	Using the notation of Proposition \ref{prop:singularities}, we can write $d_{(u,x)}G=(d_uF_x^1, d_x\varphi^{0,1}_u-\mathbbm{1})$, and since both $(u_n, x_n)\mapsto d_{u_n}F_{x_n}^1$ and $(u_n, x_n)\mapsto d_{x_n}\varphi^{0,1}_{u_n}$ are weak-strong continuous, then:
	 \be \label{eq:strong}w_i^n\stackrel{\textrm{strong}}{\longrightarrow}\overline w_i=d_{(\overline u, \overline x)}G^*e_i,\quad \forall i=1, \ldots, 2a+1.\ee
		 Notice that $T_{(u_n, x_n)}\Lambda=(W^n)^\perp$ and $\nabla_{(u_n, x_n)}J=(u_n, 0)$, hence we can decompose:
	 \be (u_n, 0)=\nabla_{(u_n, x_n)}g+\textrm{proj}_{W^n}(u_n, 0).\ee
	 By assumption $\nabla_{(u_n, x_n)}g\to 0$ and $\|u_n\|$ is bounded; moreover the fact that $\overline\gamma$ is a regular point of $\Lambda$ implies that $d_{(\overline u, \overline x)}F^*$ has maximal rank and the vectors $\{\overline w_1, \ldots, \overline w_{2a+1}\}$ in \eqref{eq:strong} form a linearly independent set. In particular (up to subsequences) we can assume that the sequence $\{\textrm{proj}_{W^n}(u_n, 0)\}_{n\in \N}$ converges to a limit $\overline w\in W.$ Putting all this together we obtain (up to subsequences):
	 \be (u_n, 0)\stackrel{\textrm{strong}}{\longrightarrow} \overline w, \ee
	 and consequently $(u_n, x_n)\to (\overline w, \overline x).$

	\subsubsection{The case $\overline \gamma$ is a singular point of $\Lambda$}In order to deal with this case we will need to use special coordinates centered at $\overline\gamma$, coordinates provided by the next lemma.
	\begin{lemma}\label{lemma:coordinates}Let $M$ be a contact sub-riemannian manifold and fix $\overline x\in M$. There exist Hilbert manifold coordinates $(u, x)$ on a neighborhood $U\simeq L^2(I, \R^{2a})\times \R^{2a+1}$ centered at the constant curve $\gamma\in \Omega$, $\gamma(t)\equiv \overline x$, such that:
	\be \label{eq:coordinates}F(u, x)=L(x,\hat F(u))\quad \textrm{and}\quad J(u,x)=\frac{1}{2}\|u\|^2+\int_{0}^1\left\langle u(t), R(\gamma_{(u,x)}(t))u(t)\right\rangle dt,\ee
	where:
	\begin{itemize}
	\item[1.] The map $L:\R^{2a+1}\times \R^{2a+1}\to \R^{2a+1}$  is affine in each variable (more precisely $L$ is the left-translation in a Heisenberg group).
	\item[2.]
	The map  $\hat{F}:L^2(I, \R^{2a})\to \R^{2a+1}$  (which corresponds to the endpoint map centered at zero for the  Heisenberg group above) is given by:
	\be\label{eq:endpoint} \hat{F}(u)=\left(\int_{0}^{1}u(s)\,ds,\int_{0}^{1}\left\langle u(s),A\int_{0}^s u(\tau)d\tau \right\rangle ds \right).\ee
	(Here $A$ is the $2a\times 2a$ skew-symmetric matrix, representing the bracket structure in the above Heisenberg group.)
	\item[3.] The map $R:\R^{2a+1}\to \textrm{Sym}(2a, \R)$ is smooth and satisfies:
	\be \label{eq:R}R(0)=0\quad \textrm{and} \quad d_0R=0.\ee
	\end{itemize}
	\begin{proof}
	By \cite[Theorem 6.9]{AgrachevGauthier2} there exists\footnote{The three-dimensional case is proved in \cite[Theorem 2.1]{AgrachevGauthier1} (see also \cite[Theorem 19]{Barilari}). In fact the statement of \cite[Theorem 6.9]{AgrachevGauthier2} claims something much stronger then what we need (the existence of a normal form), but it also makes a slightly stronger assumption (the contact structure should be ``strongly nondegenerate''). However, it follows from the proof of \cite[Theorem 6.9] {AgrachevGauthier2} that relaxing their hypothesis to general contact structures still provides a frame in the form that we need. Alternatively one can redo the proof of \cite[Theorem 2.1]{AgrachevGauthier1} for the $2a+1$ dimensional case and notice that the weaker conclusions still guarantee the existence of the required frame.} an open neighborhood $W$ with coordinates $(x, y, z)\in \R^a\times \R^a\times \R$ centered at $\overline x=(0,0,0)$ such that $\Delta|_W$ is spanned by a sub-riemannian \emph{orhonormal} frame $\{X_1, \ldots, X_a, Y_1, \ldots, Y_a\}$ of the form:
	\be X_i=\underbrace{\partial _{x_i}+\gamma_i\frac{y_i}{2} \partial_z}_{\hat X_i}+ b_{1,i} Z_{1,i}\quad \quad Y_i=\underbrace{\partial _{y_i}-\gamma_i\frac{x_i}{2} \partial_z}_{\hat Y_i}+ b_{2,i} Z_{2,i}\quad\forall i=1, \ldots, a.\ee
	For all $i=1, \ldots, a$ we have that: $0\neq \gamma_i\in \R$, $Z_{1,i}$ and $Z_{2,i}$ are bounded vector fields on $W$ and $b_{1,i}, b_{2,i}:W\to \R$ are functions such that $\label{eq:b} b_{1,i}(0)=b_{2,i}(0)=0$ and $d_{0}b_{1,i}=d_0b_{2,i}=0.$
	
	Denoting by $\hat\Delta=\textrm{span}\{\hat{X}_i, \hat{Y}_i, i=1, \ldots,a\}$, we can write (again by  \cite[Theorem 6.9] {AgrachevGauthier2}):
	\be \hat\Delta=\ker\left(\underbrace{dz-\sum_{i=1}^af_i(0)(x_i dy_i-y_idx_i)}_{\hat\omega}\right)\quad \textrm{and}\quad \Delta=\ker\left(\underbrace{dz-\sum_{i=1}^af_i(x_i dy_i-y_idx_i)}_{\omega}\right)\ee
	for some smooth functions $f_1, \ldots, f_a.$
	
	We consider now the family of differential forms:
	\be \omega_t=(1-t)\hat\omega +t\omega. \ee
	We want to build a family of diffeomorphisms $\psi_t$ of a neighborhood $W'$ of the origin, fixing the origin, such that $\psi_t^*\omega_t=\omega_0$ and 
	\be \label{eq:ordertwo}d_0\psi_1=\textrm{id}+\textrm{terms of order two}.\ee Notice that on a sufficiently small neighborhod of the origin $\omega_t$ is a contact form for all $t\in [0,1]$.
	
		In the new coordinates induced by $\psi_1$ we will have that $(\R^{2a+1}, \Delta)$ is a contact Carnot group (a Heisenberg group, see \cite{AgrachevGentileLerario} for more details on the geometry of endpoint maps for Carnot groups). This choice of coordinates on $W'$ and the trivializing frame $\{\hat X_i, \hat Y_i, i=1, \ldots, a\}$ for $\Delta|_{W'}$ define coordinates on a neighborhood of $\overline \gamma$ in the Hilbert manifold $\Omega$; the endpoint map centered at zero writes exactly as in \eqref{eq:endpoint} (see \cite[Section 2.2]{AgrachevGentileLerario}), and the endpoint map centered at a point $x$ is obtained by composition with the left-translation by $x$ in the Carnot group. 

Using \eqref{eq:ordertwo} we can write the matrix $S$, representing the scalar product on $\hat{\Delta}$ induced by pulling back the metric from $\Delta$ using $\psi_1$, as:
\be S(p)=J\psi_1(p)^TJ\psi_1(p)=\mathbbm{1}+\textrm{terms of order two}, \ee
which will imply the claim.
	
	It remains to prove the existence of such a family $\psi_t$. We use the classical Moser's trick, realizing $\psi_t$ as the flow of a non-autonomous vector field $X_t$. We adapt the proof of  \cite[Theorem 2.5.1]{Geiges}; for the rest of the proof let us use the convention that the symbol $\epsilon(k)$ denotes ``terms of order $k$'' in $p$ as $p\to 0$.
	
	Since the flow of a vector field admits the expansion:
	\be \psi_1(p)=p+\int_{0}^1X_t(p)dt+\epsilon(2),\ee
 it is clearly enough to prove that we can find $X_t$ satisfying $X_t=\epsilon(2).$ 
	
	We write $X_t=H_tR_t+Y_t$ where $R_t$ is the Reeb vector field of $\omega_t$ (which in this case just equals $\partial_z$),  $H_t$ is a smooth function and $Y_t\in \ker(\omega_t)$. We can look for $X_t\in \ker(\omega_t)$ thus setting $H_t\equiv 0$ in the proof of \cite[Theorem 2.5.1]{Geiges}. Then $\psi_t^*\omega_t=\omega_0$ simply writes:
	\be\label{eq:moser} \dot \omega_t+i_{Y_t}d\omega_t=0.\ee
	This is in turn equivalent to the pair of equations:
	\be\label{eq:eqww} (\dot \omega_t+i_{Y_t}d\omega_t)(\partial_z)=0\quad \textrm{and}\quad (\dot \omega_t+i_{Y_t}d\omega_t)|_{\ker \omega_t}=0.\ee
	By construction $i_{Y_t}d\omega_t\partial_z=\mathcal{L}_{Y_t}\omega_t(\partial_z)-d(\omega_t(Y_t))(\partial_z)=0$ and
\be\label{eq:dot}\dot\omega_t=\sum_{i=1}^a(f_i-f_i(0))(x_idy_i-y_i dx_i)=\epsilon(2).\ee
	 Consequently $R_t=\partial_z\in \ker({\dot\omega}_t)$ and thus the first equation in \eqref{eq:eqww} is automatically satisfed. The non-degeneracy of $d\omega_t|_{\ker (\omega_t)}$ implies then that we can find a unique $Y_t$ solving the above equation (as in the proof of Gray's stability Theorem, see \cite[Theorem 2.2.2]{Geiges}). Using matrix and vector notation for differential forms, there exist $\xi_t$ and $\Omega_t$ such that:
	\be \dot\omega_t|_{\ker \omega_t}(v)= \xi_t^T v\quad \textrm{and} \quad d\omega_t|_{\ker \omega_t}(v_1, v_2)=v_1^T\Omega_tv_2.\ee
	We see that \eqref{eq:moser} is equivalent to:
	\be\label{eq:finale} \Omega_tY_t= \xi_t.\ee
	Note that  $ \xi_t=\epsilon(2)$ by \eqref{eq:dot} and, since $ d\omega_t=d\omega+\epsilon(1)$, we also have $\Omega_t=\Omega_0+\epsilon(1)$.

Denoting by $P_t$ the orthogonal projection on the image of the operator $\Omega_t:\ker(\omega_t)\to \R^{2a+1}$, we can write \eqref{eq:finale} as:
\be M_tY_t=\xi_t,\quad M_t=P_t\Omega_t.\ee
Now $M_t:\ker(\omega_t)\to \textrm{im}(\Omega_t)$ is an invertible operator and by construction $M_t=M_0+\epsilon(1).$ As a consequence:
\begin{align} Y_t&=M_t^{-1}\xi_t=(M_0+\epsilon(1))^{-1}\xi_t=(M_0^{-1}+\epsilon(1))\xi_t\\
&=M_0^{-1}\xi_t+\epsilon(1)\xi_t=\epsilon(2)+\epsilon(3)=\epsilon(2).\end{align}
This proves the existence of $X_t=Y_t$ satisfying our requirements. The flow of $X_t$ is defined for all $t\in [0,1]$ at $p=0$ (simply because it fixes the origin), hence it is also defined for all $t\in [0,1]$ on a sufficiently small neighborhood $W'$ of the origin. This concludes the proof.
	\end{proof}

		\end{lemma}
Going back to the proof of Theorem \ref{thm:PS} in the case $\overline\gamma$ is a singular point of $\Lambda$, we fix the coordinates  given by Lemma \ref{lemma:coordinates} and we proceed to show the next intermediate lemma. Note that in such coordinates $J$ is not equal to the $L^{2}$ squared norm of the control, so one should not expect the usual formula $d_{(u,x)}J=u$.
\begin{lemma}\label{lem:rnto0}
Let $\gamma_{(u_n,x_n)}\weakto \bar{\gamma}$, where  $\gamma \equiv \bar{x}$, and consider the coordinates given in Lemma \ref{lemma:coordinates} centered at $\bar{x}$. Then
\be\label{eq:nabla} \nabla_{(u_n, x_n)} J=(u_n,0)+r_n\quad \textrm{with $r_n\stackrel{\textrm{strong}}{\longrightarrow}0$}.\ee 
\end{lemma}

\begin{proof}
In coordinates, the differential of $J$ can be computed as:
\begin{align}
 d_{(u,x)}J(\dot u, \dot x)=&\int_{0}^1\langle \dot u(t), u(t)\rangle dt+2\int_{0}^1\langle \dot u(t), R(\gamma_{(u,x)}(t))u(t)\rangle dt\\
&+\int_{0}^1\left\langle u(t), \left(d_{\gamma_{(u,x)}(t)}R\,d_{(u,x)}F^t (\dot u, \dot x)\right)u(t)\right\rangle dt=(*).
\end{align}
In the previous equation $F^t$ denotes the time-$t$ endpoint map $F^t(u,x)=\gamma_{(u,x)}(t)$  in this chart (such a map was previously denoted by $F^t_x(u)$, this slight abuse will simplify the notation in the next computations) and its differential can be computed as:
\be d_{(u,x)}F^t (\dot u, \dot x)=d_uF_x^t\dot u+d_x\varphi_u^{0, t}\dot x.\ee
In particular:
\begin{align}(*)=&\int_{0}^1\langle \dot u(t), u(t)\rangle dt+\underbrace{2\int_{0}^1\langle \dot u(t), R(\gamma_{(u,x)}(t))u(t)\rangle dt}_{R_1(u,x)\dot u}\\
&+\underbrace{\int_{0}^1\left\langle u(t), \left(d_{\gamma_{(u,x)}(t)}R\,d_uF_x^t\dot u\right)u(t)\right\rangle dt}_{R_2(u,x)\dot u}+\underbrace{\int_{0}^1\left\langle u(t), \left(d_{\gamma_{(u,x)}(t)}R\,d_x\varphi_u^{0, t}\dot x\right)u(t)\right\rangle dt}_{R_3(u,x)\dot x}.\end{align}
 We will prove below that if $(u_n, x_n)\rightharpoonup (0, 0)$ then the three linear operators $R_1(u_n, x_n), R_2(u_n, x_n)$ and $R_3(u_n, x_n)$ all converge to zero (strongly); this directly implies \eqref{eq:nabla}.
 
 Consider first $R_1(u_n, x_n)$. Notice that since $\gamma_{(u_n, x_n)}$ converges uniformly to $\overline \gamma$ (by Proposition \ref{prop:weak})  then for every $\epsilon>0$ there exists $n_1$ such that:
    $$ \sup_{t \in [0,1]} \|R(\gamma_{(u_n, x_n)}(t))\|_{\textrm{Sym}(2a, \R)}\leq \epsilon \quad  \text{for all }  n\geq n_1, $$
 where  $\| \cdot \|_{\textrm{Sym}(2a, \R)}$ denotes the operatorial norm in the space of the symmetric  matrices $\textrm{Sym}(2a, \R)$.  
   Then for all $n\geq n_1$ we have:
 \begin{align}\|R_1(u_n, x_n)\|&=\sup_{\|\dot u\|=1}|R_1(u_n, x_n)\dot u|\leq \sup_{\|\dot u\|=1}\int_{0}^1\left|\langle \dot u(t), R(\gamma_{(u_n,x_n)}(t))u_n(t)\rangle\right|dt\\
 &\leq \sup_{\|\dot u\|=1} \epsilon\int_{0}^1 |\dot u(t)| |u_n(t)|dt  \; \leq \sup_{\|\dot u\|=1} \epsilon \left(\int_{0}^1  |\dot u(t)|^2dt\right)^{1/2} \left(\int_{0}^1  |u_n(t)|^2dt\right)^{1/2}\\
 &\leq \epsilon \sqrt{C_0},
 \end{align}
where $C_0=\sup_n \|u_n\|< \infty$ by Banach-Steinhaus Theorem since $u_n\weakto 0$.  This proves in particular that $\|R_1(u_n, x_n)\|\to 0$.

As for the term $R_2$, notice that the uniform convergence of $\gamma_n$ implies that for every $\epsilon>0$ there exists $n_2$ such that: 
$$ \sup_{t \in [0,1]}   \|d_{\gamma_{(u_n, x_n)}(t)}R\|_{\textrm{Hom}(\R^{2a+1}, \textrm{Sym}(2a, \R))} \leq \epsilon  \quad  \text{for all }  n\geq n_2,$$
where $\| \cdot \|_{\textrm{Hom}(\R^{2a+1}, \textrm{Sym}(2a, \R))} $ denotes the norm in the space of linear maps from $\R^{2a+1}$ with values in $\textrm{Sym}(2a, \R)$.
Thus for all $n\geq n_2$ we can estimate:
\begin{align}
 \int_{0}^1\left |\left\langle u_n(t), \left(d_{\gamma_{(u_n,x_n)}(t)}R\,d_{u_n}F_{x_n}^t\dot u\right)u_n(t)\right\rangle\right| dt  &\leq \epsilon \int_{0}^1|u_n(t)|^2\, |d_{u_n}F_{x_n}^t\dot u| \, dt\\
&\label{eq:step}\leq \epsilon  \int_{0}^1|u_n(t)|^2\left(|d_0F_{0}^t\dot u|+C_1\right)dt,
\end{align}
where in the last line we have used Proposition \ref{prop:weak} to infer that  $u\mapsto d_uF_x^t$ is weakly-strongly continuous (and the convergence is uniform in $t$, see \cite[Proposition 21 and Lemma 24]{BoarottoLerario}).
On the other hand, using the expression given in \cite[Proposition 5.25]{Montgomery} for the differential of the endpoint map (at the zero control), we have: 
 $$ |d_0F_{0}^t\dot u|=\left|\int_{0}^t\sum_{i=0}^{2a} \hat{X}_i(0)\dot u_i(t)dt\right|\leq C_2 \int_{0}^t  |\dot u(t)|dt \leq C_2  \|\dot u\|.$$
Plugging the last estimate   into \eqref{eq:step}, gives:
\be \int_{0}^1\left |\left\langle u_n(t), \left(d_{\gamma_{(u_n,x_n)}(t)}R\,d_{u_n}F_{x_n}^t\dot u\right)u_n(t)\right\rangle\right| dt     \leq \epsilon  \;C_0 (C_1+C_2 \| \dot u\| ),\quad \forall n\geq n_2,  
\ee
which proves that   $\|R_2(u_n, x_n)\|\to 0$.

Concerning $R_3$ we use the fact that the maps $(u,x)\mapsto \varphi_u^{0, t}(x)$ and $(u,x)\mapsto d_x\varphi_u^{0,t}$ are weakly-strongly continuous, with uniform convergence in $t$. The first statement is just the weak-strong continuity of the endpoint map as a function of both $u$ and $x$, as in Proposition \ref{prop:weak}; the second statement follows from the fact that $F^t(u,x)$ is differentiable in $x$ (\cite[Section 2.4.1]{AgrachevSachkov}) and hence its differential $\partial_xF^{t}(x,t)$ solves the ODE obtained linearizing \eqref{eq:control}; in particular Proposition \ref{prop:weak} applies to this new control system and if $(u_n, x_n)\rightharpoonup (\overline u, \overline x)$ then the matrix $d_{x_n}\varphi_{u_n}^{0,t}$ converges to $d_{\overline x}\varphi_{\overline u}^{0,t}$ uniformly in $t$. Thus, for every $\epsilon>0$ and $n\geq n_2$ we have:
\begin{align} \|R_3(u_n,x_n)\|&=\sup_{\|\dot x\|=1}|R_3(u_n, x_n)\dot x|  \leq  \int_{0}^1\left|\left\langle u_n(t), \left(d_{\gamma_{(u_n,x_n)}(t)}R\,d_{x_n}\varphi_{u_n}^{0, t}\dot x\right)u_n(t)\right\rangle\right| dt\\
&\leq \int_{0}^1|u_n(t)|^2 \|d_{\gamma_{(u_n,x_n)}(t)}R\|_{\textrm{Hom}(\R^{2a+1}, \textrm{Sym}(2a, \R))} |d_{x_n}\varphi_{u_n}^{0, t}\dot x| \, dt\leq \epsilon \cdot C_0 C_4,
\end{align}
which shows $ \|R_3(u_n,x_n)\|\to 0$ and together with the previous two estimates finally gives \eqref{eq:nabla}.
\end{proof}

We can now conclude the proof of Theorem  \ref{thm:PS}.  
\\Consider the tangent space $T_{(u_n, x_n)}\Lambda \subset L^{2}(I, \R^{2a})\times \R^{2a+1}$ in the charts of Lemma \ref{lemma:coordinates}:
\begin{align} T_{(u_n, x_n)} \Lambda &=\{(\dot u, \dot x)\,|\, d_{(u_n, x_n)}G(\dot u, \dot x)=0\}\\
&=\left\{(\dot u, \dot x)\,\bigg|\, \frac{\partial L}{\partial x}(x_n, \hat{F}(u_n))\dot x+\frac{\partial L}{\partial y}(x_n, \hat{F}(u_n))d_{u_n}\hat{F}\dot u-\dot x=0\right\},
\end{align}
where $F$ was introduced in \eqref{eq:coordinates}.
We first claim  that $(u_n, 0)\in T_{(u_n, x_n)}\Lambda,$ i.e. that:
\be\label{eq:du} \frac{\partial L}{\partial y}(x_n, \hat{F}(u_n))d_{u_n}\hat{F}u_n=0.\ee
In order to show \eqref{eq:du}, we   compute the differential of $\hat{F}$ using equation \eqref{eq:endpoint}:
\be\label{eq:diff} d_{u}\hat F\dot u=\left(\int_{0}^{1}\dot u(s)\,ds,\int_{0}^{1}\left\langle \dot u(s),A\int_{0}^s u(\tau)d\tau \right\rangle ds+\int_{0}^{1}\left\langle u(s),A\int_{0}^s \dot u(\tau)d\tau \right\rangle ds \right).\ee
Notice now that, since the curve $\gamma_{u_n,x_n}$ is a loop, we must have $F(u_n,x_n)=x_n$. Writing this condition in  the coordinates \eqref{eq:coordinates} gives that $L(x_n, \hat{F}(u_n))=F(u_n, x_n)=x_n$, which in turn implies $\hat{F}(u_n)=0$: this is because $L(x_n, \hat{F}(u_n))=x_n$ is the element $\hat{F}(u_n)$ left-translated by $x_n$ in a Lie group where $0$ is the identity element. Writing the condition $\hat{F}(u_n)=0$ using the explicit expression \eqref{eq:endpoint} yields:
\be\label{eq:relations} \int_{0}^{1} u_{n}(s)\,ds=0\quad \textrm{and} \quad \int_{0}^{1}\left\langle u_n(s),A\int_{0}^s u_n(\tau)d\tau \right\rangle ds=0.\ee
Evaluating $d_{u_n}\hat F u_n$ using relations \eqref{eq:relations} in equation \eqref{eq:diff} with $\dot u=u_n$ implies \eqref{eq:du}.

We now use claim \eqref{eq:du} in order to conclude the proof.  Let us denote by $p_n:L^{2}(I, \R^{2a})\times \R^{2a+1}\to T_{(u_n, x_n)}\Lambda$ the orthogonal projection.  Using  the fact that $(u_n, 0)\in T_{(u_n, x_n)}\Lambda$, we rewrite equation \eqref{eq:nabla} as:
\begin{align} (u_n, 0) =p_n(u_n, 0)&= p_n(\nabla_{(u_n, x_n)} J)-p_n(r_n)\\
\label{eq:final!}&=\nabla_{(u_n, g_n)}g-p_n(r_n).
\end{align}
By assumption $\nabla_{(u_n, g_n)}g\to 0$; moreover since  $p_n$ has norm one (it  is a projection operator) and $r_n\to 0$ strongly  by Lemma \ref{eq:nabla},  also $p_n(r_n)\to 0$ strongly. Together with \eqref{eq:final!} this finally proves that $(u_n, 0)\to (0,0)$ strongly and finishes the proof of Theorem \ref{thm:PS}.
	\end{proof}
	
	\subsection{A min-max principle and the existence of a closed geodesic}	\label{sec:min-max}
	Let $\alpha \in \pi_k(\Lambda)$ and observe that an element  $f \in \alpha$ is a continuous map from $S^k$ with values into $\Lambda=\Lambda(M)$, the space of horizontal loops.  Set
	\be\label{eq:defc}
	c_{\alpha}= \inf_{f \in \alpha} \;   \sup_{\theta \in S^k} \;     J( f(\theta)).
	\ee
	The goal of the present section is to prove the following general min-max principle.
	
	\begin{thm} \label{thm:minmax}
	Let $(M,\Delta)$ be a compact contact sub-riemannian manifold, fix a class $\alpha \in \pi_k(\Lambda)$ and  consider the min-max level $c_\alpha$ defined in \eqref{eq:defc}. If $c_\alpha$ is \emph{strictly} positive,  then there exists a closed geodesic $\gamma_\alpha \in \Lambda$ realizing the min-max level, i.e. 
	$J(\gamma_\alpha)=c_\alpha$. 
	\end{thm}
		
	\begin{remark}\label{rem:c=0}
	Notice that if $c_\alpha>0$ then it must be  $\alpha \neq 0$. Indeed,  if $\alpha=0$,  then every $f \in \alpha$ is homotopic to a constant map  and  therefore $c_\alpha=0$. In this case the min-max level $c_\alpha$  is trivially realized by constant curves which, in virtue of  Definition \ref{def:closedGeod}, are not closed geodesics.
	\end{remark}
	\noindent
	For every $a>0$ let us denote:
	\be\label{eq:defLambdaa}
	\Lambda^a=\{ \gamma \in \Lambda \,: \, J(\gamma) \leq a\}, \quad \text{and} \quad  \Lambda^b_a= \Lambda^b\setminus \Lambda^a, \; \text{for } 0\leq a<b.
	\ee
	We say that $c >0$ is a \emph{critical value} if there exists a curve $\gamma$ which is a critical point for $g=J|_{\Lambda}$, i.e. a horizontal loop such that $g(\gamma)=J(\gamma)=c$ and $\nabla_{\gamma} g =0$ (since $c>0$, then $\Lambda $ is smooth near $\gamma$ and the classical definition of critical point applies, as proved in Proposition \ref{prop:CharClosedG}).
	
	In order to prove Theorem \ref{thm:minmax} we will use the following deformation  lemma, which says that if there are no critical values in the interval $[a,b]$ then we can continuously deform $\Lambda^b$ into $\Lambda^a$ without moving the elements in $\Lambda^{a/2}$.
	
	\begin{lemma}[Deformation lemma]\label{lem:deformation}
	Let $(M,\Delta)$ be a compact connected contact manifold.
	Let $0< a<b$ and assume that $g= J|_\Lambda$ has no critical values in the interval $[a,b]$. Then  there exists a homotopy  $H:[0,1]\times \Lambda^b \to \Lambda^b$ such that
	\begin{enumerate}
	\item   $H(0,\cdot)=Id_{\Lambda^b}$, 
	\item   $H(1,\gamma)\in \Lambda^a$ for every $\gamma \in \Lambda^b$,
	\item   $H(t,\gamma)= \gamma$ for every $\gamma \in \Lambda^{a/2}$ and every  $t \in [0,1]$.
	\end{enumerate}
	\end{lemma}
	
	\begin{proof}
	First of all,  since $\textrm{Sing} (\Lambda)=\Lambda^0$, we have that $\Lambda^b_0 \cap \textrm{Sing} (\Lambda)=\emptyset$ and  $\Lambda^b_0$ is a smooth submanifold of $\Omega$. Moreover, by Theorem \ref{thm:PS}, we know that  $g$ satisfies the Palais-Smale condition. Notice indeed that if  $\{\gamma_n\}_{n\in \mathbb{N}}$ is a Palais-Smale sequence as in Theorem \ref{thm:PS}, then a limit $\overline \gamma$ still has energy $0<a\leq J(\overline\gamma)\leq b$. It follows that $g|_{\Lambda^b_0}=J|_{\Lambda^b_0}$ is a $C^1$ functional on a  Hilbert manifold and one can apply the standard theory of pseudo-gradient vector fields \cite[Lemma 3.2]{Chang} (equivalently, see  \cite{AmbMalc} or \cite{Struwe}) to conclude.
	\end{proof}

	\emph{Proof of Theorem \ref{thm:minmax}}.
	First of all observe that, since by assumption $c_\alpha>0$ and $(M,\Delta)$ is contact, then every $\gamma\in g^{-1}(c_\alpha)$ is  regular in the sense of \eqref{def:G}. Note also that condition \eqref{eq:lagrange} in  Proposition \ref{prop:CharClosedG} is equivalent to $\nabla g=0$. Therefore it is enough to show that $c_\alpha$ is a critical value for $g=J|_{\Lambda}$.  Assume by contradiction the opposite. 
	\\We first claim that if $c_\alpha$ is not a critical value then:
	\be\label{eq:claimcad}
	\text{there exists $\delta>0$ such that $[c_\alpha-\delta, c_\alpha+\delta]$ does not contain any critical value.}
	\ee
	Indeed otherwise there would exist a sequence of critical values $c_n\to c_\alpha$, i.e. there would exists 
	\be\label{eq:unPS}
	\{\gamma_n\}_{n \in \N} \in \Lambda \quad \text{such that } \nabla_{\gamma_n} g=0 \text{ and } g(\gamma_n)=J(\gamma_n)=c_n\to c_\alpha. 
	\ee 
	By the Palais-Smale property proved in Theorem \ref{thm:PS} we get that $\gamma_n$ strongly converges, up to subsequences, to a limit $\bar{\gamma}\in \Lambda$.  But since the functional $g:\Lambda \to \R$ is $C^1$, the properties  \eqref{eq:unPS}  force $\bar{\gamma}$ to be a critical point of $g$ with energy $c_\alpha$. This proves \eqref{eq:claimcad}, since we are assuming $c_\alpha$ not to be a critical value.
\\Now we can apply the Deformation Lemma  \ref{lem:deformation} with $a=[c_\alpha-\delta]$, $b=[c_\alpha+\delta]$, get the homotopy $H$ deforming $\Lambda^{c_\alpha+\delta}$ into $\Lambda^{c_\alpha-\delta}$ and show a contradiction with the definition of $c_\alpha$. To this aim,  call $\eta:\Lambda^{c_\alpha+\delta}\to \Lambda^{c_\alpha-\delta}$ the deformation defined by $\eta(\gamma)=H(1,\gamma)$, and observe that by the very definition \eqref{eq:defc} of $c_\alpha$, there exists $\bar{f}\in \alpha$ with $\bar{f}(S^k)\subset \Lambda^{c_\alpha+\delta}$.  Therefore $\eta\circ \bar{f}$ is still an element of the homotopy class $\alpha$ (since homotopic to $\bar{f}$ via $H$) but now  $(\eta\circ \bar{f})(S^k)\subset \Lambda^{c_\alpha-\delta}$.
\\It follows  that
$$\inf_{f\in \alpha} \sup_{\theta \in S^k} J(f(\theta))\leq   \sup_{\theta\in S^k}  J\big((\eta \circ \bar{f})(\theta)\big) \leq c_\alpha-\delta,  $$
contradicting the definition of $c_\alpha$.    
	\hfill$\Box$
	
\medskip

We can now prove the  main result of this section, namely the existence of a closed geodesic.  Such a result extends the classical and celebrated Theorem of  Lyusternik-Fet  \cite{LF} to the case of contact manifolds (for a self-contained proof in the Riemannian case, the interested reader can see \cite{Struwe, Jost, Oancea}). The proof  involves all the tools developed in the paper and follows  from the combination of the  min-max Theorem	\ref{thm:minmax} and the homotopy properties of the loop spaces established in Section \ref{ss:deformHom}.

	\begin{thm}\label{thm:ExMain}
	Let $(M,\Delta)$ be a compact, contact sub-riemannian manifold.  Then there exists at least one non constant closed  sub-riemannian geodesic.
	\end{thm}
	
	\begin{proof}
	First of all, if $\pi_1(M)\neq \{0\}$ then the claim follows by Theorem \ref{thm:pi1}. Notice that in this case the proof was achieved by a minimization procedure. 
	
	If instead $\pi_1(M)=\{0\}$, i.e. if $M$ is simply connected, then a minimization procedure would trivialize and give just a constant curve. To handle this case we then argue via  min-max: thanks to Theorem \ref{thm:minmax} it is enough to show that the min-max level $c_\alpha$ defined in \eqref{eq:defc} is strictly positive, for some $k\in \N$ and $\alpha \in \pi_k(\Lambda)$. 
	To this aim recall that given a compact $n$-dimensional manifold there exists at least one number $1\leq k \leq n$ such that $\pi_k(M)\neq 0$ (see \cite{Spanier} for a proof). Let $k\geq 2$ be the minimal one with this property, i.e.  $k\in \N$ is smallest number such that $\pi_k(M)\neq \{0\}$ but $\pi_{k-1}(M)=\{0\}$.  But then Corollary \ref{cor:nonzero} implies that for any $0\neq \alpha \in \pi_{k-1}(\Lambda)$, the corresponding min-max value $c_\alpha$ is strictly positive, concluding the proof.
	\end{proof}


\begin{thebibliography}{9}
\bibitem{AgrachevSachkov} A. A. Agrachev, Y. Sachkov, \emph{Control Theory from the Geometric Viewpoint}, Encyclopaedia of Mathematical Sciences, Vol. 87. Control Theory and Optimization, II. Springer-Verlag, Berlin, (2004).
\bibitem{AgrachevBarilariBoscain} A.~A.~Agrachev, U.~Boscain, D.~Barilari, \emph{Introduction to Riemannian and sub-riemannian geometry},  (version 17.11.17), available at: https://webusers.imj-prg.fr/~davide.barilari/2017-11-17-ABB.pdf.
To be published by Cambridge University Press 
with the title ``A Comprehensive Introduction to Sub-Riemannian Geometry".
\bibitem{AgrachevGauthier1} A. A. Agrachev, J-P. Gauthier, \emph{On the Dido problem and plane isoperimetric problem}, Acta Appl. Math., Vol. 57, (1999),  no. 3, pp. 287--338. 
\bibitem{AgrachevGauthier2} A. A. Agrachev, J-P. Gauthier, \emph{Sub-riemannian metrics and isoperimetric problems in the contact case}, Journal of Mathematical Sciences, Vol. 103, No. 6, (2001). Russian version: Geometric control theory (Moscow, 1998), 5?48, Itogi Nauki Tekh. Ser. Sovrem. Mat. Prilozh. Temat. Obz., 64, Vseross. Inst. Nauchn. i Tekhn. Inform. (VINITI), Moscow, (1999). 
\bibitem{AgrachevGentileLerario} A. A. Agrachev, A. Gentile, A. Lerario, \emph{Geodesics and horizontal path spaces in Carnot groups}, Geometry \& Topology, Vol. 19, no. 3, (2015), DOI: 10.2140/gt.2015.19.1569.
\bibitem{AmbMalc} A. Ambrosetti, A. Malchiodi, \emph{Nonlinear analysis and semilinear elliptic problems.} Cambridge Studies in Advanced Mathematics, 104. Cambridge University Press, Cambridge, (2007).
     \bibitem{Barilari} D. Barilari, \emph{Trace heat kernel asymptotics in 3D contact sub-riemannian geometry}, J.  Math. Sciences, Vol. 195,  no. 3, (2013),  pp. 391--411. 
\bibitem{BoarottoLerario} F. Boarotto, A. Lerario, \emph{Homotopy properties of endpoint maps and a theorem of Serre ins sub-riemannian geometry},   Comm. Anal. Geom., Vol. 25, no. 2,  (2017), 269--301. 
\bibitem{Bryant} R. L. Bryant, L. Hsu, \emph{Rigidity of integral curves of rank 2 distributions}, Invent. Math., Vol. 114, no. 2, (1993), pp. 435-461.

\bibitem{Chang} K.C.  Chang,  \emph{Infinite-dimensional Morse theory and multiple solution problems}. Progress in Nonlinear Differential Equations and their Applications, 6. Birkh\"auser Boston, Inc., Boston, MA, (1993).
\bibitem{CJT}Y. Chitour, F. Jean,  E. Tr\'elat, \emph{Genericity results for singular curves}, J. Diff. Geom.,  Vol. 3, no.1, (2006), pp. 45--73.

\bibitem{dynamic} J. Dominy, H. Rabitz, \emph{Dynamic Homotopy and Landscape Dynamical Set Topology in
Quantum Control}, J. Math. Phys.,  Vol. 53, no. 8, (2012).
\bibitem{EES}   T. Ekholm, J. Etnyre, M. G. Sullivan,  \emph{The contact homology of Legendrian submanifolds in $\R^{2n+1}$}, J. Diff. Geom., Vol. 71, no. 2, (2005),  pp. 177--305.
\bibitem{Geiges}H. J. Geiges, \emph{An Introduction to Contact Topology}, Cambridge studies in advanced mathematics
\bibitem{Gray} J. W. Gray, \emph{Some Global Properties of Contact Structures}, Annals of Math.
Second Series, Vol. 69, no. 2, (1959), pp. 421-450.
\bibitem{BabZhit} B. Jakubczyk,  M. Zhitomirskii, \emph{Distributions of corank 1 and their characteristic vector fields}, Trans. Am. Math. Soc.,  Vol. 355, (2003), pp. 2857--2883.
\bibitem{Jost} J. Jost, \emph{Riemannian geometry and geometric analysis. Sixth edition.} Universitext. Springer, Heidelberg, (2011). 
\bibitem{Hurewicz} W. Hurewicz, \emph{On the concept of fiber space}, Proc Nat. Acad Sci. U S A., Vol. 41, (1955), pp. 956--961. 
\bibitem{LiuSuss}W. S. Liu,  H. Sussmann, \emph{Shortest paths for sub-riemannian metrics on rank-2 distributions}, Memoirs  Am. Math. Soc.,  no. 564,  (1995).
\bibitem{LF} L. A. Lyusternik, A. I. Fet, \emph{Variational problems on closed manifolds}, Doklady
Akad. Nauk SSSR (N.S.), Vol. 81, (1951), pp. 17--18.   
\bibitem{MilnorCW}J. W. Milnor, \emph{On spaces having the homotopy type of a CW-complex}, Trans. Amer. Math. Soc., Vol. 90, (1959), pp.  272--280.
\bibitem{Montgomery}R. Montgomery, \emph{A Tour of sub-riemannian Geometries, Their Geodesics and Applications},
Mathematical Surveys and Monographs, Vol. 91. Am.  Math.  Soc., Providence, RI, (2002). 
\bibitem{Oancea} A. Oancea, \emph{Morse theory, closed geodesics, and the homology of free loop spaces}. With an appendix by Umberto Hryniewicz. IRMA Lect. Math. Theor. Phys., 24, Free loop spaces in geometry and topology, 67--109, Eur. Math. Soc., Z¸rich, (2015).
	\bibitem{rure}R. Fritsch, R. Piccinini, \emph{Cellular Structures in Topology},
Cambridge Studies in Advanced Mathematics, Vol. 19, (1990). 
	\bibitem{Sarychev}A. V. Sarychev, \emph{On homotopy properties of the space of trajectories of a completely
nonholonomic differential system}, Soviet Math. Dokl., Vol. 42,  (1991), pp. 674--678.

\bibitem{Sabloff} J. M. Sabloff, \emph{What is...a Legendrian Knot?}, Notices Amer. Math. Soc., Vol. 56, no. 10, (2009), pp. 1282--1284.
\bibitem{Spanier} H. Spanier, \emph{Algebraic Topology. Corrected reprint.} Springer-Verlag, New York-Berlin, (1981).
\bibitem{Struwe} M. Struwe, \emph{Variational methods. Applications to nonlinear partial differential equations and Hamiltonian systems. Fourth edition.} A Series of Modern Surveys in Mathematics, Vol. 34, Springer-Verlag, Berlin, (2008).

\end{thebibliography}
\end{document}